\documentclass{amsart}
\usepackage[utf8]{inputenc}

\usepackage{amsmath,amssymb,mathrsfs,amsthm}
\usepackage{mathtools} 
\usepackage[usenames]{color}
\usepackage{color}  
\usepackage{graphicx} 
\usepackage{tabularx}
\usepackage{supertabular} 
\usepackage{enumerate} 
\usepackage{url}
\usepackage{tikz} 
\usepackage{pgfplots}  
\usetikzlibrary{intersections, calc} 
\usetikzlibrary{patterns}  
\usetikzlibrary{arrows,decorations.markings} 
\usetikzlibrary {lindenmayersystems}

\newcommand{\bs}[1]{\mbox{$\boldsymbol{#1}$}} 

\pgfplotsset{compat=1.18}

\newtheorem{theorem}{Theorem}[section]

\newtheorem{lemma}[theorem]{Lemma}
\newtheorem{proposition}[theorem]{Proposition}
\newtheorem{remark}[theorem]{Remark}
\newtheorem{definition}[theorem]{Definition}
\newtheorem{corollary}[theorem]{Corollary}
\newtheorem{example}[theorem]{Example}

\title[Young integration on self-similar set in intervals]
{A Young-type integration on self-similar sets in intervals} 

\author[T. MARUYAMA]{Takashi MARUYAMA}
\address{NEC Laboratories Europe, Kurfürsten-Anlage 36, 69115 Heidelberg, Germany}
\email{Takashi.Maruyama@neclab.eu}

\author[T. SETO]{Tatsuki SETO}
\address{General Education and Research Center, Meiji Pharmaceutical University, 
2-522-1 Noshio, Kiyose-shi, Tokyo, Japan}
\email{tatsukis@my-pharm.ac.jp}


\begin{document}

\begin{abstract} 
We introduce a generalization of the Young integration on 
self-similar sets defined in a closed interval  
and give a sufficient condition of its integrability. 
We also prove 
integration by substitution, 
integration by parts 
and term-by-term integration 
and give examples of the properties. 
\end{abstract}

\maketitle

\section*{Introduction}  
Integration is used in a wide range of fields. 
One representative integration is the integration by the Lebesgue measure, which serves as a fundamental ingredient in expanding on the theory of the Riemann integration. 
The Lebesgue integration has abundant applications for sets with positive measures, such as open sets or their closure in $\mathbb{R}^{n}$. 
On the other hand, the Lebesgue integration on a set with the Lebesgue measure $0$ 
cannot have favorable applications since the Lebesgue integration of any functions on such a set is trivial.   
Examples of a set of the measure $0$ can be found in fractal sets.  
Indeed, the middle-third Cantor set, a typical example of fractal sets, 
has the measure $0$, 
which motivates one to seek other types of integration for fractal sets. 
 
For fractal sets, some works employ the Hausdorff measure to define integration on fractal sets. In the context of noncommutative geometry, there are many studies that relate the Dixmier trace with the Hausdorff probability measure on self-similar sets; see for example \cite{MR2357322,MR1303779,MR3743228,MR4651742}. 
There is also a work \cite{MR4562886} which introduces single and double integrals using Hausdorff measure and conducts numerical comparison of the integrands with the ground-truth value. 

Apart from the integrations mentioned above, the Young integration \cite{MR1555421} is another generalization of the Riemann integration on $1$-dimensional smooth compact manifolds. 
The Young integral is defined as (the limit of) the Riemann sum of the point-wise difference of functions on subdivisions of a given interval $I$.
As a corollary of the convergence of the limit, 
the Young integration on $I$  induces a bilinear functional from the product space $C^{\alpha}(I) \times C^{\beta}(I)$ of the H\"{o}lder continuous classes for $0 < \alpha, \beta \leq 1$ with $\alpha + \beta > 1$ to the complex numbers 
\[
Y \colon 
C^{\alpha}(I) \times C^{\beta}(I) \to \mathbb{C}. 
\]
The Young integration on an interval can be also extended on a circle.

The definition of the Young integration would fit a class of fractal sets such as self-similar sets defined by iterated function systems (IFSs) built on a closed interval (resp. a circle) since IFSs generate the union of contracted intervals (resp. circles).    
Indeed, there is a study \cite{phdmaruyama} that 
defines a cyclic cocycle based on the Young integration on 
the union of infinitely many closed curves or circles
by 
introducing a new class of self-similar sets; note that the class does not contain self-similar sets in intervals such as the middle-third Cantor set.   
In a talk \cite{talk:Moriyoshi2013}, a variant of the Young integration was defined on a new class of algebra of functions, 
that is the pullback of the continuous and bounded variation class on a unit circle by 
the Cantor function restricted on the middle-third Cantor set. 
There is also a study \cite{MR4688447} which generalizes the result of \cite{talk:Moriyoshi2013} to ``two dimension''.   
While the literature \cite{MR4688447,talk:Moriyoshi2013} defined a Young-type integration on non-differentiable spaces, they assume to use structures of smooth manifolds by connecting the disconnected components of self-similar sets with the Cantor function.  
In other words, the integration is not the one natively defined on self-similar sets.  

In the present paper, we define and study a Young-type integration on a class of self-similar sets in an interval $I$. 
The class includes intervals and the middle-third Cantor set as examples. 
As opposed to the literature \cite{MR4688447,talk:Moriyoshi2013}, the definition of our integration solely relies on the information of IFS. 

Let $(I,S,\{ f_{s} \})$ be IFS on $I$ and
$K$ the self-similar set of $(I,S,\{ f_{s} \})$. 
We define a variant of the Riemann sum $\phi_{n}(f,g)$ on $\cup_{s_{1},\dots , s_{n}} (f_{s_{1}} \circ \cdots \circ f_{s_{n}})(I)$ for any $n \in \mathbb{N}$ and functions $f,g$ on $K$. 
Taking the limit $\phi(f,g) = \lim_{n \to \infty}\phi_{n}(f,g)$ as $n \to \infty$, and we define a Young-type integration on $K$; see for the details in Section \ref{sec:young}. 
We show the integrability condition of the H\"{o}lder continuous functions in Section \ref{sec:convergence}, 
in which we also give various examples of integrable pairs and  their corresponding integral, and non-integrable pairs in Section \ref{subsec:2k+1-adic} and \ref{subsec:integrability}.  
In Section \ref{subsec:2k+1-adic}, we need generalizations of the Cantor function; the definition and properties of the generalizations 
are explained in Appendix \ref{sec:generalized-Cantor-function}. 
We further investigate properties of our integration:  
integration by substitution (Section \ref{sec:substitution}), 
integration by parts (Section \ref{sec:parts}) and 
term-by-term integration (Section \ref{sec:term-by-term}). 
Here, we outline our integration by substitution  
since the property reveals an explicit connection of our integration to the variant of the Young integration studied in \cite{talk:Moriyoshi2013}. 
Let $\mathcal{L}^{(i)} = \left(I^{(i)}, S, \{ f^{(i)}_{s} \} \right)$ be IFS on an interval $I^{(i)}$ 
and $K^{(i)}$ the associated self-similar set of $\mathcal{L}^{(i)}$ for $i = 1, 2$. 
If a permutation of $S$ with a condition induces a map $T \colon K^{(1)} \to K^{(2)}$, 
we can prove integration by substitution 
$\phi^{(1)}(f \circ T, g \circ T) = \pm \phi^{(2)}(f,g)$, where the sign $\pm$ depends only on the permutation but does not coincide with the signature. 
Note that the condition of permutations can not be removed; see Section \ref{sec:not_hold_substitution}. 
In particular, if $K^{(1)}$ equals the middle-third Cantor set and $K^{(2)}$ equals the unit interval $[0,1]$, 
the idea using 
the pullback of functions on an interval by the Cantor function introduced in \cite{talk:Moriyoshi2013} can be interpreted as an example of integration by substitution of our integration when the permutation equals identity; see Section \ref{subsec:MN-substitution}. 
In Section \ref{subsec:MN-substitution}, we also explain the concrete idea introduced in the talk \cite{talk:Moriyoshi2013} for the sake of readability.

We note that in order to generalize our results to higher dimension, 
we may use the Fredholm module introduced in \cite{MR4651742} and need more general methods. 
We will leave the study for future work.

\section{Preliminaries}
\label{sec:prelim} 

We firstly recall the Fredholm module 
on an interval introduced by A. Connes \cite{MR1303779}.  
Let $I = [a,b]$ (resp. $I = (a,b)$) 
be the bounded closed (resp. open) interval, 
$C(\{ a,b \})$ the $C^{\ast}$-algebra of the continuous functions 
equipped with sup norm $\| \cdot \|$ 
on the boundary $\partial I = \{ a,b \}$, 
and $\ell^{2}(\{ a,b \})$ the $\ell^{2}$-space on $\partial I = \{ a,b \}$. 
Then $C(\{ a,b \})$ acts on $\ell^{2}(\{ a,b \})$ 
by the pointwise multiplication. 
An identification  
$\ell^{2}(\{ a,b \}) = \ell^{2}(a) \oplus \ell^{2}(b) 
\cong \mathbb{C}^{2}$ gives rise to the action of   
the matrix $F = \begin{bmatrix} 0 & 1 \\ 1 & 0 \end{bmatrix}$  on $\ell^{2}(\{ a,b \})$. 
Similarly, 
all $2 \times 2$ matrices acts on $\ell^{2}(\{ a,b \})$. 
Under the identification, 
we consider $\mathcal{B}(\ell^{2}(\{ a  , b  \})) = M_{2}(\mathbb{C})$. 
Then the triple $(C(\{ a,b \}) , \ell^{2}(\{ a , b \}), F)$ 
is a Fredholm module, and we call the triple the Fredholm module on $I$.

Next we recall the definition of self-similar sets in the interval $I = [a,b]$. 
Let $N$ be a positive integer and set $S = \{ 0,1, \dots , N-1 \}$. 
Let 
$f_{s} : I \to I$ 
($s \in S$) 
be similitudes with the similarity ratio  
\[
r_{s} = \dfrac{|f_{s}(x) - f_{s}(y)| }{| x - y |} \;\,
(<1)
\quad (x \neq y) . 
\]   
We assume $r_{s} > 0$ for all $s \in S$. 
An iterated function system (IFS)
$(I, S , \{ f_{s} \}_{s \in S})$ 
induces the unique non-empty compact set 
$K$  
called the self-similar set 
such that 
$K = \bigcup_{s=0}^{N-1} f_{s}(K)$. 
We denote by $\dim_{S}(K)$ the similarity dimension of $K$, 
that is,  
the unique solution $d$ of the equation 
\[ 
\sum_{s=0}^{N-1}r_{s}^{d} = 1. 
\]  

We here introduce some notation used throughout this paper. 
Take a tuple $\bs{s} = (s_{1}, \dots ,  s_{n}) \in S^{\infty} = \bigcup_{n=0}^{\infty} S^{\times n}$, 
where we assume that $S^{\times 0} = \{ \emptyset \}$. 
We define $f_{\bs{s}} = f_{s_{1}} \circ \dots \circ f_{s_{n}}$ 
and 
$f_{\emptyset} = \mathrm{id}$. 
We set $I_{\bs{s}} = f_{\bs{s}}(I)$, 
and it is immediate to show that $I_{(\bs{s} , t)}$ ($t \in S$) is a subset in $I_{\bs{s}}$.  
Then we have 
$K = \cap_{n = 1}^{\infty} \cup_{\bs{s} \in S^{\times n}}  I_{\bs{s}}$.

Let us denote 
$a_{\bs{s}} = f_{\bs{s}}(a)$ and 
$b_{\bs{s}} = f_{\bs{s}}(b)$ 
for any $\bs{s} \in S^{\infty}$. 
Then we have 
$I_{\bs{s}} = \left[ a_{\bs{s}} , b_{\bs{s}} \right]$. 
In this paper, we allow ``few 
overlaps'' between any pair of $I_{s} = f_{s}(I)$ ($s = 0,1,\dots , N-1$), 
that is, 
we assume the inequality $\sharp \left( I_{s_{1}} \cap I_{s_{2}} \right) \leq 1$ 
if $s_{1} \neq s_{2}$. 
For the simplicity, we assume that the images $f_{s}(I)$ are ordered from left to right with respect to the indices $s$, that is, 
\[
a = a_{0} < b_{0} \leq a_{1} < b_{1} \leq \cdots \leq a_{N-1} < b_{N-1} = b.  
\]  
Under this assumption, we get   
\[
a_{\bs{s}} = a_{(\bs{s},0)} < b_{(\bs{s},0)} \leq a_{(\bs{s},1 )} < b_{(\bs{s} ,1)} \leq \cdots \leq a_{( \bs{s} ,N^{n}-1)} < b_{( \bs{s} , N^{n}-1 )} = b_{\bs{s}} 
\] 
for any $\bs{s} \in S^{\infty}$.

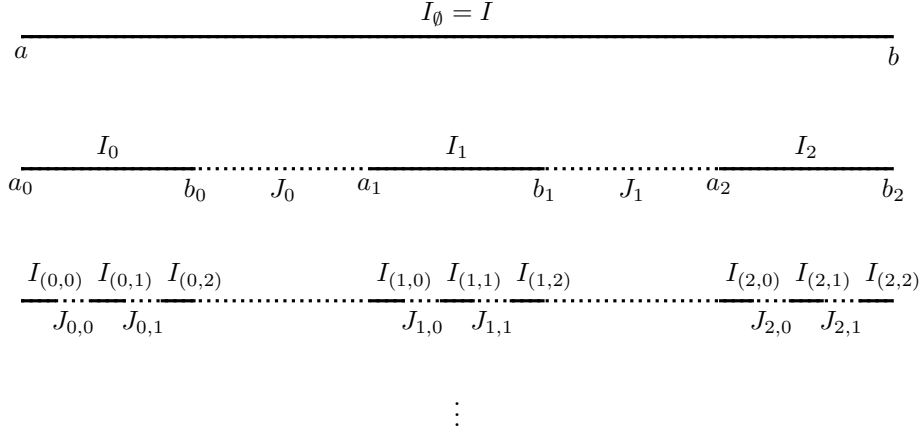
\begin{figure}[h]
\centering 
\pgfdeclarelindenmayersystem{cantor set}{
  \rule{F -> FfFfF}
  \rule{f -> fffff} 
}
\pgfdeclarelindenmayersystem{open cantor}{
  \rule{F -> FFFFF}
  \rule{f -> fFfFf} 
}
\begin{tikzpicture}[very thick]  
\def \length{330pt};   
\def \shift{50pt};  
\foreach \order in {0,...,2}
\draw[dotted,yshift=-\order*\shift]  l-system[l-system={open cantor, axiom=F, order=\order, step=\length/(5^\order)}]  ; 
\draw (3*\length/10,-\shift) node[below]{$J_{0}$}; 
\draw (7*\length/10,-\shift) node[below]{$J_{1}$}; 
\draw (3*\length/50,-2*\shift) node[below]{$J_{0,0}$}; 
\draw (7*\length/50,-2*\shift) node[below]{$J_{0,1}$}; 
\draw (23*\length/50,-2*\shift) node[below]{$J_{1,0}$}; 
\draw (27*\length/50,-2*\shift) node[below]{$J_{1,1}$}; 
\draw (43*\length/50,-2*\shift) node[below]{$J_{2,0}$}; 
\draw (47*\length/50,-2*\shift) node[below]{$J_{2,1}$}; 
\foreach \order in {0,...,2}
  \draw[yshift=-\order*\shift]  
  l-system[l-system={cantor set, axiom=F, order=\order, step=\length/(5^\order)}]  ; 
\draw (0,0) node[below]{$a$}; 
\draw (\length,0) node[below]{$b$}; 
\draw (\length/2,0) node[above]{$I_{\emptyset} = I$}; 
\draw (0,-\shift) node[below]{$a_{0}$}; 
\draw (\length/5,-\shift) node[below]{$b_{0}$}; 
\draw (2*\length/5,-\shift) node[below]{$a_{1}$}; 
\draw (3*\length/5,-\shift) node[below]{$b_{1}$}; 
\draw (4*\length/5,-\shift) node[below]{$a_{2}$};  
\draw (\length,-\shift) node[below]{$b_{2}$}; 
\draw (\length/10,-\shift) node[above]{$I_{0}$}; 
\draw (5*\length/10,-\shift) node[above]{$I_{1}$}; 
\draw (9*\length/10,-\shift) node[above]{$I_{2}$}; 
\draw (1*\length/5^2,-2*\shift) node[above]{$I_{(0,0)}$}; 
\draw (3*\length/5^2,-2*\shift) node[above]{$I_{(0,1)}$};  
\draw (5*\length/5^2,-2*\shift) node[above]{$I_{(0,2)}$}; 
\draw (11*\length/5^2,-2*\shift) node[above]{$I_{(1,0)}$}; 
\draw (13*\length/5^2,-2*\shift) node[above]{$I_{(1,1)}$};  
\draw (15*\length/5^2,-2*\shift) node[above]{$I_{(1,2)}$}; 
\draw (21*\length/5^2,-2*\shift) node[above]{$I_{(2,0)}$}; 
\draw (23*\length/5^2,-2*\shift) node[above]{$I_{(2,1)}$};  
\draw (25*\length/5^2,-2*\shift) node[above]{$I_{(2,2)}$};  
\node (D) at (\length/2,{-3*\shift + 10pt}) {$\vdots$}; 
\end{tikzpicture}  
\caption{Labels of intervals}
\end{figure} 

Let $J_{\bs{s},l}$ ($\bs{s} \in S^{\infty}$, $l = 0,1,\dots , N-2$) 
be open intervals of 
the ``holes'' of $I_{\bs{s}} = f_{\bs{s}}(I)$ 
labeled by the ordering from left to right, 
that is,  
we define 
\begin{equation*}
J_{\bs{s},0} = \left( b_{(\bs{s},0)} , a_{(\bs{s},1)} \right), \     
J_{\bs{s},1} = \left( b_{(\bs{s},1)} , a_{(\bs{s},2)} \right),  \ 
\cdots, \ 
J_{\bs{s},N-2} = \left( b_{(\bs{s},N-2)} , a_{(\bs{s},N-1)} \right). 
\end{equation*} 
We also use simpler notation $J_{l} = J_{\bs{s}, l}$ 
if there is no confusion. 
Then  we have 
$K = I \setminus \left( \cup_{\bs{s} \in S^{\infty}} \cup_{l=0}^{N-2} J_{\bs{s},l}\right)$.

A typical example of  IFSs in the interval is the middle-third Cantor set. Its construction is as follows:  
Let $I = [0,1]$ and define an IFS 
$\{ f_{0} , f_{1} \colon I \to I \}$  by  
\[
f_{0}(x) = \dfrac{1}{3}x, \quad
f_{1}(x) = \dfrac{1}{3}x + \dfrac{2}{3}.  
\]
The middle-third Cantor set $CS$ is defined as the unique 
non-empty compact set of $I$ induced by this IFS and satisfies 
$\displaystyle CS =  f_{0}(CS) \cup f_{1}(CS)$. 
The followings are some examples of the notations for $\displaystyle CS$ introduced above:  
\begin{align*} 
I_{0} &= f_{0}(I) = \left[ 0, \dfrac{1}{3}\right], & 
I_{1} &= f_{1}(I) = \left[ \dfrac{2}{3} , 1 \right], & 
J_{0} &= \left( \dfrac{1}{3} , \dfrac{2}{3} \right), \\ 
I_{(0,0)} &= f_{(0,0)}(I) = \left[ 0, \dfrac{1}{9} \right], &
I_{(0,1)} &= f_{(0,1)}(I) = \left[ \dfrac{2}{9}, \dfrac{3}{9} \right], & 
J_{0,0} &= \left( \dfrac{1}{9} , \dfrac{2}{9} \right), \\ 
I_{(1,0)} &= f_{(1,0)}(I) = \left[ \dfrac{6}{9} , \dfrac{7}{9} \right], &
I_{(1,1)} &= f_{(1,1)}(I) = \left[ \dfrac{8}{9} , 1 \right], & 
J_{1,0} &= \left( \dfrac{7}{9} , \dfrac{8}{9} \right) . 
\end{align*} 
With the notation introduced above, we have 
\[
CS 
= \bigcap_{n=1}^{\infty} \bigcup_{\bs{s} \in \{ 0,1 \}^{\times n}} I_{\bs{s}}
= I \setminus \left( \bigcup_{\bs{s} \in \{ 0,1 \}^{\infty}}  J_{\bs{s},0} \right). 
\] 

\section{Definition of a Young-type integration}
\label{sec:young} 

In this section, we define a Young-type integration 
on $K$ 
and prove a vanishing result for H\"{o}lder continuous functions with a sufficiently large H\"{o}lder exponent. Let  $(C(\{ a  , b  \}) , \ell^{2}(\{ a  , b  \}) , F_{I} )$ be the Fredholm module on $I  = \left[ a  , b  \right]$ (or $I = (a,b)$). 
We define 
\[
M_{I} = -\dfrac{1}{  b-a}[F_{I} , x] = 
\begin{bmatrix} 0 & -1 \\ 1 & 0 \end{bmatrix}. 
\] 
For the simplicity, we omit the subscript $I$ 
if there is no confusion. 

\begin{lemma} 
\label{lem:trace_riemannsum}
For any continuous function $f,g \in C( \{ a,b \} )$, we have 
\[
\mathrm{Tr}(f[F, g]M) = (f(a) + f(b))(g(b) - g(a)). 
\]
\end{lemma} 

\begin{proof}
Since we have 
\[
[F,g] = (g(b) - g(a))\begin{bmatrix} 0 & 1 \\ -1 & 0 \end{bmatrix},   
\]
we obtain the lemma.  
\end{proof} 
Lemma \ref{lem:trace_riemannsum} indicates that 
the trace of an operator 
$f[F, g]M$ gives 
rise to a discretized version of the Young integration on intervals.  
See also Appendix \ref{app:Young} for the definition of the Young integration on an interval. 

From now on, we denote 
\[
\mathrm{Tr}_{I}(f[F,g]M) 
= \mathrm{Tr}(f[F_{I} , g]M_{I})   
\] 
for an bounded interval $I$ and two functions $f,g$ on $\partial I$. 
We define the integrability of a Young-type integration 
on $K$. 

\begin{definition} 
\label{def:integrable} 
For any functions $f$ and $g$ on $K$ and 
$n \in \mathbb{N}$, we define 
\[
\phi_{n}(f,g) 
= \sum_{\bs{s} \in S^{\times n}} \mathrm{Tr}_{I_{\bs{s}}}(f[F,g]M).
\] 
We call a pair $(f,g)$ is integrable  on $K$ 
if the sequence $\{ \phi_{n}(f,g) \}$ is convergent. 
We denote the limit by $\phi(f,g)$ and call it the integral of $(f,g)$ on $K$. 
\end{definition}

We first show a class of functions such that the integral vanishes.  
Denote by $C^{\alpha}(K)$ the $\alpha$-H\"{o}lder continuous functions on $K$ and set 
\[ 
|f|_{\alpha} = \sup_{x \neq y} \dfrac{|f(x) - f(y)|}{|x-y|^{\alpha}} \quad (f \in C^{\alpha}(K)).
\]

\begin{theorem} 
\label{thm:vanish}  
For any functions $f \in C(K)$ and $g \in C^{\beta}(K)$ with $\beta > \dim_{S}(K)$, 
we have $\phi (f,g) = 0$.  
\end{theorem} 

\begin{proof}   
If $\beta > \dim_{S}(K)$, we have 
\begin{align*} 
|\phi_{n}(f,g)| 
&\leq \sum_{\bs{s} \in S^{\times n}} |\mathrm{Tr}_{f_{\bs{s}}(I)}(f[F,g]M) | 
\leq \sum_{\bs{s} \in S^{\times n}} |f(a_{\bs{s}}) + f(b_{\bs{s}})| |g(b_{\bs{s}}) - g(a_{\bs{s}})| \\ 
&\leq 2\| f \| |g|_{\beta} \sum_{\bs{s} \in S^{\times n}} |b_{\bs{s}} - a_{\bs{s}}|^{\beta} 
= 2(b-a)^{\beta}\| f \| |g|_{\beta} \sum_{\bs{s} \in S^{\times n}} \prod_{i=1}^{n} r_{s_{i}}^{\beta} \\ 
&= 2(b-a)^{\beta}\| f \| |g|_{\beta} \left( \sum_{s=0}^{N-1} r_{s}^{\beta}\right)^{n} 
\to 0 \quad (n \to \infty). 
\end{align*}  
\end{proof}

We note that 
our integration on intervals 
does not coincide 
with either of the Young integration or the Lebesgue integration. 

\begin{example} 
\label{exm:interval}
We first consider the correspondence of our integration to the Young integration; see Appendix \ref{app:Young} for the Young integration.  
Let $I = [a,b]$ be the closed interval and we set 
\[
f_{0}(x) = \dfrac{1}{2}(x-a) + a  \ \text{and} \  
f_{1}(x) = \dfrac{1}{2}\left( x - b \right) +  b 
\quad 
(x \in I). 
\] 
Then a tuple $(I, \{ 0,1 \}, \{ f_{0}, f_{1} \})$  
defines an IFS and the self-similar set of the IFS equals $I$. 
Take two functions $f,g$ on $I$ such that the Stieltjes integral $Y(f,g)$ 
exists in Riemann sense. 
By the definition of our integral, 
the pair $(f,g)$ is integrable  
and we have 
\[
\phi (f,g) = 2\int_{a}^{b}f(x) dg(x). 
\]
On the other hand, let $f$ be the Dirichlet function on $I$. 
Then a pair $(f,x)$ is integrable and its integral satisfies $\phi(f,x) = 2$, 
but the Stieljes integral $Y(f,x)$ does not exist.  
Therefore, our integral is different from the both of the Young integration and Lebesgue integration. 
\end{example}

\section{Convergence of integration}
\label{sec:convergence}

In this section, we prove a sufficient condition  for the existence of the limit of $\phi_{n}$ in Theorem \ref{thm:iff_condition}. 
We also calculate examples of our integration.  

\subsection{Sufficient condition for the integrability} 
In order to prove a sufficient condition for the integrability, 
we firstly prove an easy lemma.  
Set $\Delta_{I} f = f(b) - f(a)$ for any function $f$ on $\partial I = \{ a,b \}$. 

\begin{lemma} 
\label{lem:2bunkatsu}
Let $I =[a,b]$, $x \in (a,b)$, $A = [a,x]$ and $B = [x,b]$. 
\begin{enumerate} 
\item 
For any function $g$ on $\{ a,b,x \}$, we have 
$[F_{I}, g] - [F_{A} , g] = [F_{B}, g]$ 
in $M_{2}(\mathbb{C})$. 
\item 
For any function $f, g$ on $\{ a,b,x \}$, we have 
\[
\operatorname{Tr}_{A}(f[F,g]M) + \operatorname{Tr}_{B}(f[F,g]M) 
= \operatorname{Tr}_{I}(f[F,g]M) + \Delta_{A}f \Delta_{B}g - \Delta_{B}f \Delta_{A}g. 
\]
\end{enumerate} 
\end{lemma} 

\begin{proof}  
\begin{enumerate} 
\item 
The following easy calculation 
proves this part. 
\begin{align*} 
[F_{I}, g] - [F_{A} , g] 
&= 
\begin{bmatrix} 
 & g(b) - g(a) \\ 
-(g(b) - g(a)) & 
\end{bmatrix} 
- 
\begin{bmatrix} 
 & g(x) - g(a) \\ 
-(g(x) - g(a)) & 
\end{bmatrix} \\ 
&= 
\begin{bmatrix} 
 & g(b) - g(x) \\ 
-(g(b) - g(x)) & 
\end{bmatrix} 
= 
[F_{B} , g]. 
\end{align*} 
\item  
Denote by $P_{i}$ the projection on to 
the $i$-th component of $\mathbb{C}^{2}$. 
By  (1), we have 
\begin{align*} 
&\phantom{=} 
\operatorname{Tr}_{A}(f[F,g]M) + \operatorname{Tr}_{B}(f[F,g]M) - \operatorname{Tr}_{I}(f[F,g]M) \\ 
&= 
\operatorname{Tr}(f(a)P_{1}[F_{A},g]M) + \operatorname{Tr}(f(x)P_{2}[F_{A},g]M) \\ 
&\hspace*{10mm}+ \operatorname{Tr}(f(x)P_{1}[F_{B},g]M) + \operatorname{Tr}(f(b)P_{2}[F_{B},g]M) \\ 
&\hspace*{10mm}- \operatorname{Tr}(f(a)P_{1}[F_{I},g]M) - \operatorname{Tr}(f(b)P_{2}[F_{I},g]M) f\\ 
&= 
-\operatorname{Tr}(f(a)P_{1}[F_{B},g]M) - \operatorname{Tr}(f(b)P_{2}[F_{A},g]M) \\ 
&\hspace*{10mm}+ \operatorname{Tr}(f(x)P_{2}[F_{A},g]M) + \operatorname{Tr}(f(x)P_{1}[F_{B},g]M) \\ 
&= 
(f(x) - f(a))\operatorname{Tr}(P_{1}[F_{B},g]M)
-(f(b) - f(x))\operatorname{Tr}(P_{2}[F_{A},g]M) \\ 
&= 
(f(x) - f(a))(f(b) - f(x)) 
- (f(b) - f(x))(f(x) - f(a)) \\ 
&= 
\Delta_{A}f \Delta_{B}g - \Delta_{B} f \Delta_{A} g. 
\end{align*} 
This proves our lemma. 
\end{enumerate} 
\end{proof} 

By recursively applying  Lemma \ref{lem:2bunkatsu} (2), 
we can generalize Lemma \ref{lem:2bunkatsu} (2) 
for any number of pieces of divisions of $I$. 
For example, let $A = [a,x]$, $B = [x,y]$ and $C = [y,b]$. 
Then we have 
\begin{align*} 
&\phantom{=} 
\operatorname{Tr}_{A}(f[F,g]M) + \operatorname{Tr}_{B}(f[F,g]M) + \operatorname{Tr}_{C}(f[F,g]M) \\ 
&= 
\operatorname{Tr}_{A \cup B}(f[F,g]M) + \operatorname{Tr}_{C}(f[F,g]M) + \Delta_{A}g \Delta_{b} g - \Delta_{B}f \Delta_{A} g \\ 
&= 
\operatorname{Tr}_{A \cup B \cup C}(f[F,g]M) 
+ \Delta_{A} f \Delta_{B} g
+ \Delta_{A \cup B}f \Delta_{C} g 
- \Delta_{B} f \Delta_{A} g 
- \Delta_{C}f \Delta_{A \cup B}g. 
\end{align*} 

For the proof of a sufficient condition for the integrability, we define another sequence on $\cup_{\bs{s} \in S^{\times n}} \cup_{l=0}^{N-2} J_{\bs{s},l}$: 
\begin{align*}
\psi_{n+1}(f,g) 
&= \sum_{\bs{s} \in S^{\times n}}\sum_{l=0}^{N - 2} \ \mathrm{Tr}_{J_{\bs{s},l}}(f[F,g]M) . 
\end{align*}

\begin{theorem}
\label{thm:iff_condition} 
Let $f \in C^{\alpha}(K)$ and $g \in C^{\beta}(K)$ with 
$\alpha + \beta > \dim_{S}(K)$. 
The pair $(f,g)$ is integrable on $K$
if and only if 
the series $\sum_{n=1}^{\infty} \psi_{n}(f,g)$ converges. 
\end{theorem}

\begin{proof}    
In a manner similar to Section \ref{sec:prelim}, we denote $(\bs{s} , k) = (s_{1}, \dots , s_{n} , k) \in S^{\times (n+1)}$ for $\bs{s} = (s_{1}, \dots , s_{n})\in S^{\times n}$ and $k \in S$. 
For the simplicity, we will use notations $I_{k} = I_{(\bs{s} , k)}$ and $J_{l} = J_{\bs{s}, l}$ as long as it is clear what they mean. Then, 
by recursively applying Lemma \ref{lem:2bunkatsu} (2), we get 
\begin{align*} 
&\phantom{=}\sum_{k=0}^{N-1} \operatorname{Tr}_{I_{(\bs{s}, k)}}(f,[F,g]M) 
+ 
\sum_{l=0}^{N-2} \operatorname{Tr}_{J_{\bs{s}, l}}(f,[F,g]M) 
- 
\operatorname{Tr}_{I_{\bs{s}}}(f,[F,g]M) \\ 
&= 
\left( \Delta_{I_{0}}f \Delta_{J_{0}}g + \Delta_{I_{0} \cup J_{0}} f \Delta_{I_{1}}g + \cdots + \Delta_{I_{0} \cup \dots \cup J_{N-2}}f \Delta_{I_{N-1}}g \right) \\ 
&\hspace*{10mm}- 
\left( \Delta_{J_{0}}f\Delta_{I_{0}}g + \Delta_{I_{0}}f \Delta_{I_{0} \cup J_{0}}g + \cdots + \Delta_{I_{N-1}}f \Delta_{I_{0} \cup \cdots \cup J_{N-2}}g \right). 
\end{align*}

\noindent 
With this expression, we have 
\begin{align*} 
&\phantom{=}\phi_{n+1}(f,g) - \phi_{n}(f,g)  \\ 
&= 
\sum_{\bs{s} \in S^{\times n}} \sum_{k=0}^{N-1} \operatorname{Tr}_{I_{(\bs{s}, k)}}(f,[F,g]M) 
- 
\sum_{\bs{s} \in S^{\times n}} 
\operatorname{Tr}_{I_{\bs{s}}}(f,[F,g]M)  \\ 
&= 
\sum_{\bs{s} \in S^{\times n}} 
\left( \Delta_{I_{0}}f \Delta_{J_{0}}g + \Delta_{I_{0} \cup J_{0}} f \Delta_{I_{1}}g + \cdots + \Delta_{I_{0} \cup \dots \cup J_{N-2}}f \Delta_{I_{N-1}}g \right) \\ 
&\hspace*{10mm}- 
\sum_{\bs{s} \in S^{\times n}} 
\left( \Delta_{J_{0}}f\Delta_{I_{0}}g + \Delta_{I_{0}}f \Delta_{I_{0} \cup J_{0}}g + \cdots + \Delta_{I_{N-1}}f \Delta_{I_{0} \cup \cdots \cup J_{N-2}}g \right) \\ 
&\hspace*{10mm}-
\psi_{n+1}(f,g). 
\end{align*} 

\noindent 
Then we have 
\begin{align*} 
&\phantom{=}\phi_{n+m}(f,g) - \phi_{n}(f,g) \\ 
&= 
\sum_{r=0}^{m-1}\left( \phi_{n+r+1}(f,g) - \phi_{n+r}(f,g)  \right) \\ 
&= 
\sum_{r=0}^{m-1}
\sum_{\bs{s} \in S^{\times (n+r)}} 
\left( \Delta_{I_{0}}f \Delta_{J_{0}}g + \Delta_{I_{0} \cup J_{0}} f \Delta_{I_{1}}g + \cdots + \Delta_{I_{0} \cup \dots \cup J_{N-2}}f \Delta_{I_{N-1}}g \right) \\ 
&\hspace*{10mm}- 
\sum_{r=0}^{m-1}
\sum_{\bs{s} \in S^{\times (n+r)}} 
\left( \Delta_{J_{0}}f\Delta_{I_{0}}g + \Delta_{I_{0}}f \Delta_{I_{0} \cup J_{0}}g + \cdots + \Delta_{I_{N-1}}f \Delta_{I_{0} \cup \cdots \cup J_{N-2}}g \right) \\ 
&\hspace*{10mm}-
\sum_{r=0}^{m-1}
\psi_{n+r}(f,g). 
\end{align*}

By the H\"{o}lder continuity of $f,g$, we have  
\begin{align*} 
&\phantom{=}
\left| \Delta_{I_{0}}f \Delta_{J_{0}}g + \Delta_{I_{0} \cup J_{0}}f \Delta_{I_{1}}g + \cdots + \Delta_{I_{0} \cup \dots \cup J_{N-2}}f \Delta_{I_{N-1}}g \right| \\ 
&\leq 
\left| \Delta_{I_{0}}f \Delta_{J_{0}}g \right| 
+ \left| \Delta_{I_{0} \cup J_{0}}f \Delta_{I_{1}}g \right| 
+ \cdots + 
\left| \Delta_{I_{0} \cup \dots \cup J_{N-2}}f \Delta_{I_{N-1}}g  \right|  \\ 
&\leq 
|f|_{\alpha}|g|_{\beta} (2N+1) (b-a) \prod_{i=1}^{N}r_{s_{i}}^{\alpha + \beta} 
\shortintertext{and} 
&\phantom{=}
\left| \Delta_{J_{0}}f\Delta_{I_{0}}g + \Delta_{I_{0}}f \Delta_{I_{0} \cup J_{0}}g + \cdots + \Delta_{I_{N-1}}f \Delta_{I_{0} \cup \cdots \cup J_{N-2}}g \right| \\ 
&\leq 
\left| \Delta_{I_{0}}f \Delta_{J_{0}}g \right| 
+ \left| \Delta_{I_{0}}f \Delta_{I_{0} \cup J_{0}}g \right| 
+ \cdots + 
\left|  \Delta_{I_{N-1}}f \Delta_{I_{0} \cup \cdots \cup J_{N-2}}g   \right| \\ 
&\leq 
|f|_{\alpha}|g|_{\beta} (2N+1) (b-a) \prod_{i=1}^{N}r_{s_{i}}^{\alpha + \beta}.    
\end{align*}

Therefore, for any $m \in \mathbb{N}$ and $\alpha + \beta > \dim_{S}(K)$, we have 
\begin{align*} 
&\phantom{=}\left| 
\sum_{r=0}^{m-1}
\sum_{\bs{s} \in S^{\times (n+r)}} 
\left( \Delta_{I_{0}}f \Delta_{J_{0}}g + \Delta_{I_{0} \cup J_{0}} f \Delta_{I_{1}}g + \cdots + \Delta_{I_{0} \cup \dots \cup J_{N-2}}f \Delta_{I_{N-1}}g \right) \right.\\ 
&\hspace*{10mm}- 
\left. 
\sum_{r=0}^{m-1}
\sum_{\bs{s} \in S^{\times (n+r)}} 
\left( \Delta_{J_{0}}f\Delta_{I_{0}}g + \Delta_{I_{0}}f \Delta_{I_{0} \cup J_{0}}g + \cdots + \Delta_{I_{N-1}}f \Delta_{I_{0} \cup \cdots \cup J_{N-2}}g \right) \right| \\ 
&\leq
2|f|_{\alpha}|g|_{\beta} (2N+1) (b-a) \sum_{r=0}^{\infty}
\sum_{\bs{s} \in S^{\times (n+r)}} \prod_{i=1}^{N}r_{s_{i}}^{\alpha + \beta}  \\ 
&=
2|f|_{\alpha}|g|_{\beta} (2N+1) (b-a) \sum_{r=0}^{\infty}
\left( \sum_{s=0}^{N-1} r_{s}^{\alpha + \beta} \right)^{n+r} \\ 
&= 2|f|_{\alpha}|g|_{\beta} (2N+1) (b-a) \left( \sum_{s=0}^{N-1} r_{s}^{\alpha + \beta} \right)^{n} \sum_{r=0}^{\infty} \left( \sum_{s=0}^{N-1} r_{s}^{\alpha + \beta} \right)^{r} \\ 
&\to 0 \quad (n \to \infty). 
\end{align*} 
Therefore, 
the sequence $\{ \phi_{n}(f,g) \}$ is Cauchy 
if and only if 
$\sum_{r=0}^{m-1}\psi_{n+r}(f,g) \to 0$ ($n \to \infty$). 
Thus we get the conclusion.

\end{proof}

Theorem \ref{thm:iff_condition} implies 
the following reasonable sufficient condition 
for the  integrability.

\begin{corollary} 
\label{cor:star-condition}
Let $f \in C^{\alpha}(K)$ and $g \in C^{\beta}(K)$ with 
$\alpha + \beta > \dim_{S}(K)$. 
Assume that there exists $N \in \mathbb{N}$ 
such that 
$g$ is constant on $\partial J_{\bs{s}, l}$
for any $\bs{s} \in S^{\times n}$ ($n \geq N$) and $l = 0,1,\dots , N-2$. 
Then a pair $(f,g)$ is integrable. 
\end{corollary}

\subsection{Examples on the $(2k+1)$-adic Cantor set}
\label{subsec:2k+1-adic}
We take $(2k+1)$-adic Cantor set $CS_{k}$ and generalizations of the Cantor function as examples and study the integrability of some class of functions through our integration $\phi$. 
The definition and some properties of $CS_{k}$ and the generalized Cantor function  are detailed in Appendix \ref{sec:generalized-Cantor-function}. 

For $k \in \mathbb{N}$ and $0 < p < 1$, 
let $c_{k,p}$ be a generalized Cantor function.  
Then,  we can show conditions for the integrability and non-integrability of pairs $(1,c_{k,p})$ and $(c_{k,p} , c_{k,q})$ 
as well as the numerical results of their integration.

\begin{theorem} 
\label{thm:integration-c_{k,p}} 
For any $0< p, q < 1$, we have the followings: 
\begin{enumerate} 
\item 
For $0 < p \leq 1/(k+1)$, 
a pair $(1,c_{k,p})$ is integrable 
and its integral is 
\begin{equation*} 
\phi(1,c_{k,p}) 
= 
\begin{cases} 
0 & \text{for}\ 0 < p < 1/(k+1) \\ 
2 & \text{for}\ p=1/(k+1)
\end{cases}. 
\end{equation*}  
On the other hand, we have $\phi(1,c_{k,p}) = \infty$ 
for any $1/(k+1) < p < 1$. 
\item 
For $0 < q \leq 1/(k+1)$, 
a pair $(c_{k,p}, c_{k,q})$ is integrable 
and its integral is 
\begin{equation*} 
\phi(c_{k,p}, c_{k,q}) 
= 
\begin{cases} 
0 & \text{for}\ 0 < q < 1/(k+1)\\ 
1 & \text{for}\ q=1/(k+1)
\end{cases}. 
\end{equation*}  
On the other hand, we have $\phi(c_{k,p}, c_{k,q}) = \infty$ 
for any $1/(k+1) < q < 1$. 
\end{enumerate} 
\end{theorem} 

\begin{proof} 
Let $\bs{s} \in S^{\times n}$. 
Then we have 
\begin{align*} 
c_{k,p}(b_{\bs{s}}) - c_{k,p}(a_{\bs{s}}) 
&= 
\dfrac{1-p}{k} \left( \sum_{i=1}^{n} s_{i}p^{i-1}  +  \sum_{i=n+1}^{\infty} kp^{i-1} \right) 
-
\dfrac{1-p}{k}\sum_{i=1}^{n} s_{i}p^{i-1} \\ 
&= 
(1-p)\sum_{i=n+1}^{\infty} p^{i-1} 
= 
p^{n}. 
\end{align*} 
\begin{enumerate} 
\item 
For any $n \in \mathbb{N}$, we have 
\begin{align*} 
\phi_{n}(1, c_{k,p}) 
&= 
2\sum_{\bs{s} \in S^{\times n}} \left( c_{k,p}(b_{\bs{s}}) - c_{k,p}(a_{\bs{s}}) \right)   
= 
2\sum_{\bs{s} \in S^{\times n}} p^{n} \\
&= 
2\left( (k+1)p \right)^{n} 
\longrightarrow 
\begin{cases} 
0 & \text{for}\ 0 < p < 1/(k+1) \\ 
2 & \text{for}\ p = 1/(k+1) \\ 
\infty & \text{for}\ 1/(k+1) < p < 1 
\end{cases} \quad (n \to \infty). 
\end{align*} 
This proves (1).  
\item 
For any $\bs{s} \in S^{\times n}$, 
we denote $\bs{s}' = (k-s_{1}, k-s_{2}, \dots , k-s_{n})$. 
We have 
\begin{align*} 
\phi_{n}(c_{k,p}, c_{k,q}) 
&= 
\sum_{\bs{s} \in S^{\times n}} \left( c_{k,p}(a_{\bs{s}}) + c_{k,p}(b_{\bs{s}}) \right) \left( c_{k,q}(b_{\bs{s}}) - c_{k,q}(a_{\bs{s}}) \right)  \\   
&= 
\sum_{\bs{s} \in S^{\times n}} \left( c_{k,p}(a_{\bs{s}}) + c_{k,p}(b_{\bs{s}}) \right) q^{n} \\  
&= 
q^{n}\sum_{\bs{s} \in S^{\times n}} \left( c_{k,p}(a_{\bs{s}}) + c_{k,p}(b_{\bs{s}'}) \right)  \\  
&= 
q^{n} 
\dfrac{1-p}{k}
\sum_{\bs{s} \in S^{\times n}} \left( \sum_{i=1}^{n} s_{i}p^{i-1} 
+ \sum_{i=1}^{n} (k-s_{i})p^{i-1} + \sum_{i=n+1}^{\infty} k p^{i-1}  \right) 
\\ 
&= 
(1-p)
q^{n}\sum_{\bs{s} \in S^{\times n}} \left( \sum_{i=1}^{\infty} p^{i-1}  \right) 
=  
q^{n}\sum_{\bs{s} \in S^{\times n}} 1  \\ 
&= 
\left( (k+1)q \right)^{n} 
\longrightarrow 
\begin{cases} 
0 & \text{for}\ 0 < q < 1/(k+1) \\ 
1 & \text{for}\ q = 1/(k+1) \\ 
\infty & \text{for}\ 1/(k+1) < q < 1 
\end{cases} \quad (n \to \infty). 
\end{align*} 
This proves (2).  
\end{enumerate} 
\end{proof} 

By Theorem \ref{thm:c_{k,a}-Holder}, 
if $pq < 1/(k+1)$, then 
the H\"{o}lder exponent of $c_{k,p}$ 
plus 
the H\"{o}lder exponent of $c_{k,q}$ 
is greater than $\dim_{S}(CS_{k})$. 
So this means that by Theorem \ref{thm:iff_condition},
the series $\sum_{n=1}^{\infty} \psi_{n}(c_{k,p} , c_{k,q})$ 
diverges the case when $pq < 1/(k+1)$ and $1/(k+1) < q < 1$. 
We can also prove it directly. 
In fact, for any $n = 0,1,2,\dots$, 
we have 
\begin{align*} 
\psi_{n+1} (c_{k,p} , c_{k,q}) 
&= 
\sum_{\bs{s} \in S^{\times n}} \sum_{l=0}^{k-1} \left( c_{k,p}(b_{(\bs{s},l)}) + c_{k,p}(a_{(\bs{s},l+1)}) \right) \left( c_{k,p}(a_{(\bs{s},l+1)}) - c_{k,p}(b_{(\bs{s},l)}) \right) \\ 
&= 
\sum_{\bs{s} \in S^{\times n}} \sum_{l=0}^{k-1} \left( c_{k,p}(b_{(\bs{s},l)}) + c_{k,p}(a_{(\bs{s},l+1)}) \right) \dfrac{1-q}{k} \left( q^{n} - \dfrac{k}{1-q}q^{n+1} \right) \\ 
&= 
\left( \dfrac{1-q}{k} - q \right)q^{n}\sum_{\bs{s} \in S^{\times n}} \sum_{l=0}^{k-1} \left( c_{k,p}(b_{(\bs{s},l)}) + c_{k,p}(a_{(\bs{s},l+1)}) \right) \\ 
&= 
\left( \dfrac{1-q}{k} - q \right)q^{n}\sum_{\bs{s} \in S^{\times n}} \sum_{l=0}^{k-1} \left( c_{k,p}(b_{(\bs{s},l)'}) + c_{k,p}(a_{(\bs{s},l+1)}) \right) \\ 
&= 
\left( \dfrac{1-q}{k} - q \right)q^{n}\sum_{\bs{s} \in S^{\times n}} \sum_{l=0}^{k-1} \left( 1 + \dfrac{p^{n}}{k} - \dfrac{k+1}{k}p^{n+1} \right) \\ 
&= 
\left( 1 + \dfrac{p^{n}}{k} - \dfrac{k+1}{k}p^{n+1} \right) 
\left( \dfrac{1-q}{k} - q \right)q^{n}k(k+1)^{n}. 
\end{align*} 
This implies that 
the series $\sum_{n=1}^{\infty} \psi_{n}(c_{k,p} , c_{k,q})$ 
diverges for $1/(k+1) \leq pq$ or $1/(k+1) < q < 1$. 
Figure \ref{fig:integrability} presents integrability of $(c_{k,p} , c_{k,q})$. 
Here, (i) and (ii) mean  
\begin{enumerate}[(i)] 
\item 
the area of too small H\"{o}lder exponents, and  
\item 
the area in which the series $\sum \psi_{n}$ diverges. But it has sufficiently big H\"{o}lder exponents. 
\end{enumerate} 

\begin{figure}[h] 
\begin{tikzpicture}[scale = 2.2]
\def \k{2}; 
\def \a{0.25}; 
\def \d{0.05}; 
\draw[->,>=stealth,thick] (-0.1,0)--(1.3,0)node[right]{$p$};  
\draw[->,>=stealth,thick] (0,-0.1)--(0,1.3)node[above]{$q$}; 
\coordinate[label = below left:O] (O) at (0,0); 
\coordinate[label = below:$1$] (A) at (1,0); 
\coordinate[label = left:$1$] (B) at (0,1); 
\coordinate[label = below:{\small $1/(k+1)$}] (C1) at ({1/(\k + 1)},0); 
\coordinate[label = left:{\small $1/(k+1)$}] (C2) at (0,{1/(\k + 1)});  
\draw[dotted] (A) -- ($(A) + (B)$) -- (B); 
\draw[dotted] (C1) -- ($(C1) + (B)$); 
\draw[dotted] (C2) -- ($(C2) + (A)$);  
\draw[samples = 50 , domain = {1/(\k+1)-0.07}:{1+0.3}] plot (\x,{1/((\k+1)*\x)})node[above right]{$pq = 1/(k+1)$};   
\node at (1/2, {1/(2*(\k + 1))}) {integrable}; 
\draw[fill, color = white] ({(2.3)/(\k + 1)}, {(2.3)/(\k + 1)}) circle (0.12) node[color = black]{(i)};  
\draw[fill, color = white] ({(0.9)/(\k + 1)}, {(1.8)/(\k + 1)}) circle (0.12) node[color = black]{(ii)};  
\end{tikzpicture}
\caption{Integrability of $(c_{k,p} , c_{k,q})$.} 
\label{fig:integrability}
\end{figure}
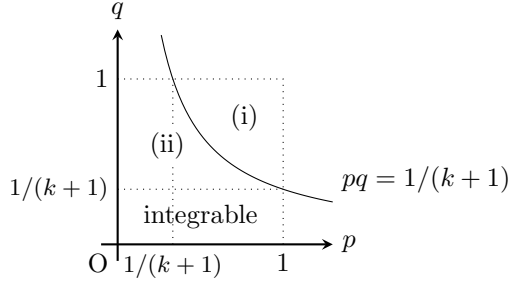 

\subsection{Integral of polynomials} 
\label{subsec:integrability} 
In this section,  
we focus on the integration of polynomials 
defined on the middle third Cantor set $CS$. 
More precisely, we calculate $\phi(x^{m}, c)$ for 
non-negative integers $m=0,1,2,\dots$, 
where we denote $c : CS \to \mathbb{R}$ the restriction of the Cantor function onto $CS$. 
Note that the case of $m=0$ is already calculated in Theorem \ref{thm:integration-c_{k,p}} (1) 
and 
that of $m=1$ is also in Theorem \ref{thm:integration-c_{k,p}} (2); 
see also Proposition \ref{prop:c_{k,a}-property}.

\begin{example} 
\label{phi_cantor}
We have  
\begin{align*}
\phi(x^{m}, c) = \begin{dcases*}
    \, 2 & for $m=0$, \\ 
    \frac{1}{2} \sum_{r=0}^{m-1} \left(- \frac{1}{3} \right)^{r} \binom{m}{r} \phi(x^{r}, c) & for $m \colon$odd, \\
    \frac{3^{m}}{2 \left(3^{m} - 1\right)} \sum_{r=0}^{m-1} \left(- \frac{1}{3} \right)^{r} \binom{m}{r} \phi(x^{r}, c) & for $m \colon$even. 
\end{dcases*}
\end{align*} 
\end{example}

\begin{proof}  
We now prove our claims for $m \geq 1$.
For each $n \in \mathbb{N}$, we take the union of boundary points of $I_{n, k}$ for all $k \in \{0, 1, \cdots, 2^{n+1} - 1\}$. 
Let $A_{n}$ be the set of numerators of the 
fractions representing the boundary points whose denominator is $3^{n}$.  
Then, $A_{n}$ may be written as  
\[
A_{n} = A_{n-1} \cup \{3^{n} - a\ ;\ a \in A_{n-1}\}.
\] 
Here we set $A_{0} = \{0, 1\}$ for convention. 
Define $S_{n}^{(m)}$ to be $\sum_{a \in A_{n}} a^{m}$, 
and we have 
\begin{align}
  \phi_{n}(x^{m}, c) &= \frac{1}{2^{n}} \sum_{k=0}^{2^{m}-1} ((a_{n,k})^{m} + (b_{n,k})^{m})  
      = \frac{1}{2^{n}}\frac{S_{n}^{(m)}}{3^{nm}}. \label{phimn}
\end{align}
Since 
\begin{align*}
(3^{n} - a)^{m} 
&= \sum_{r=0}^{m}(-1)^{r} \binom{m}{r} (3^{n})^{m-r}a^{r} 
= \sum_{r=0}^{m}(-1)^{r} \binom{m}{r} 3^{nm-nr}a^{r},
\end{align*}
we have 
\[
S_{n}^{(m)} 
=
\sum_{a \in A_{n-1}} a^{m} + \sum_{a \in A_{n-1}} (3^{n} - a)^{m}
= 
S_{n-1}^{(m)} + \sum_{r=0}^{m}(-1)^{r} \binom{m}{r} 3^{nm-nr} S_{m-1}^{(r)}.
\]  
By dividing the formula by $2^{n}\cdot3^{nm}$, we have
\begin{align*}
\frac{S_{n}^{(m)}}{2^{n}\cdot3^{nm}} &= \frac{S_{n-1}^{(m)}}{2^{n}\cdot3^{nm}} + \sum_{r=0}^{m}(-1)^{r} \binom{m}{r} 3^{nm-nr} \frac{S_{n-1}^{(r)}}{2^{n}\cdot3^{nm}}\\
&=\frac{1}{2\cdot 3^{m}} \frac{S_{n-1}^{(m)}}{2^{n-1}\cdot3^{m(n-1)}} + \sum_{r=0}^{m}(-1)^{r} \binom{m}{r} \frac{1}{2\cdot 3^{r}} \frac{S_{n-1}^{(r)}}{2^{n-1}\cdot3^{r(n-1)}}.
\end{align*}
By the equation (\ref{phimn}), the above last equation may be written as 
\[
\phi_{n}(x^{m}, c) = \frac{1}{2\cdot 3^{m}} \phi_{n-1}(x^{m}, c) + \frac{1}{2}\sum_{r=0}^{m}\left(-\frac{1}{3}\right)^{r} \binom{m}{r} \phi_{n-1}(x^{r}, c).
\]
Therefore, as $n \rightarrow \infty$, we get 
\begin{align*}
\phi(x^{m}, c) 
&= \frac{1}{2\cdot 3^{m}} \phi(x^{m}, c) + \frac{1}{2}\sum_{r=0}^{m}\left(-\frac{1}{3}\right)^{r} \binom{m}{r} \phi(x^{r}, c) \\ 
&= 
\begin{dcases} 
\frac{1}{2}\sum_{r=0}^{m-1}\left(-\frac{1}{3}\right)^{r} \binom{m}{r} \phi(x^{r}, c) & \text{for}\ m \colon \text{odd}, \\ 
\frac{1}{3^{m}} \phi(x^{m}, c) + \frac{1}{2}\sum_{r=0}^{m-1}\left(-\frac{1}{3}\right)^{r} \binom{m}{r} \phi(x^{r}, c) & \text{for}\ m \colon \text{even}. 
\end{dcases} 
\end{align*} 
This completes the proof.  
\end{proof}

Here, we give some explicit calculation for small $m$ in Example \ref{phi_cantor}.

$m = 1 \colon \phi(x, c) = 1$. 

$m = 2 \colon $
\begin{align*}
\phi(x^{2}, c) 
&=\frac{9}{16} \left(\phi(x^{0}, c) - \frac{2}{3} \phi(x^{1}, c) \right) 
= \frac{9}{16} \left(2 - \frac{2}{3} \right) 
= \frac{3}{4}.
\end{align*}

$m = 3 \colon $
\begin{align*}
\phi(x^{3}, c)  
&= \frac{1}{2} \left(\phi(x^{0}, c) - \frac{1}{3} \cdot 3 \cdot \phi(x^{1}, c) + \frac{1}{3^{2}} \cdot 3 \cdot \phi(x^{2}, c) \right) \\
&= \frac{1}{2} \left(2 - 1 + \frac{1}{3} \cdot \frac{3}{4} \right) = \frac{5}{8}.
\end{align*}

$m = 4 \colon $
\begin{align*}
\phi(x^{4}, c) 
&= \frac{3^{4}}{2 \left(3^{4} - 1\right)} \left(\phi(x^{0}, c) - \frac{4}{3} \cdot \phi(x^{1}, c) + \frac{6}{3^{2}} \cdot \phi(x^{2}, c) - \frac{4}{3^{3}}\cdot \phi(x^{3}, c) \right) \\
&= \frac{3^{4}}{2 \left(3^{4} - 1\right)} \left(2 - \frac{4}{3} + \frac{2}{3} \cdot \frac{3}{4} - \frac{4}{3^{3}} \cdot \frac{5}{8} \right) = \frac{3}{2 \cdot 80} \cdot 29 = \frac{87}{160}.
\end{align*}

$m = 5 \colon $
\begin{align*}
\phi(x^{5}, c) 
&= \frac{1}{2} \left(\phi(x^{0}, c) - \frac{5}{3} \cdot \phi(x^{1}, c) + \frac{10}{3^{2}} \cdot \phi(x^{2}, c) - \frac{10}{3^{3}} \cdot \phi(x^{3}, c) + \frac{5}{3^{4}} \cdot \phi(x^{4}, c) \right) \\
&= \frac{1}{2} \left(2 - \frac{5}{3} + \frac{10}{3^{2}} \cdot \frac{3}{4} - \frac{10}{3^{3}} \cdot \frac{5}{8} + \frac{5}{3^{4}} \cdot \frac{87}{160} \right) = \frac{31}{64}.
\end{align*}

$m = 6 \colon $
\begin{align*}
\phi(x^{6}, c) 
&= \frac{3^{6}}{2 \left(3^{6} - 1\right)} \left( \phi(x^{0}, c) - \frac{6}{3} \cdot \phi(x^{1}, c) + \frac{15}{3^{2}} \cdot \phi(x^{2}, c) \right. \\
&\ \hspace{30mm} \left. - \frac{20}{3^{3}} \cdot \phi(x^{3}, c) + \frac{15}{3^{4}} \cdot \phi(x^{4}, c) - \frac{6}{3^{4}} \cdot \phi(x^{5}, c) \right) \\
&= \frac{729}{2 \cdot 728} \left(2 - \frac{6}{3}  + \frac{15}{3^{2}} \cdot \frac{3}{4} - \frac{20}{3^{3}} \cdot \frac{5}{8} + \frac{15}{3^{4}} \cdot \frac{87}{160} - \frac{6}{3^{4}} \cdot \frac{31}{64} \right) \\ 
&= \frac{5278905}{11927552}.
\end{align*}

\section{Integration by substitution} 
\label{sec:substitution}

\subsection{Statement of integration by substitution} 
\label{subsec:statement-substitution}

We study integration by substitution for our integration  
by using a modified version of isomorphism  
of self-similar structure of IFS; 
see also \cite[Definition 1.3.2]{MR1840042}.  
Recall that our self-similar structures satisfying  
the images $f_{s}(I)$ are ordered from left to right with respect to the indices $s$, that is, 
\[
a = a_{0} < b_{0} \leq a_{1} < b_{1} \leq \cdots \leq a_{N-1} < b_{N-1} = b.  
\]

Set 
$S^{(1)} = S^{(2)} = \{ 0,1,\dots, N-1 \}$. 
Let $\mathcal{L}^{(i)} = \left(I^{(i)}, S^{(i)}, \{ f^{(i)}_{s} \}_{s \in S^{(i)}} \right)$ be IFS on an interval $I^{(i)} = \left[ a^{(i)}, b^{(i)} \right]$ 
and $K^{(i)}$ the associated self-similar set of $\mathcal{L}^{(i)}$ for $i = 1, 2$.  
There exists the unique 
continuous surjection $\pi^{(i)} \colon \Sigma(S^{(i)}) \to K^{(i)}$, 
where $\Sigma(S^{(i)}) = (S^{(i)})^{\mathbb{N}}$ 
is the shift space with symbols $S^{(i)}$; 
see also \cite[Definition 1.3.3]{MR1840042}.   
For a permutation $\rho \colon S^{(1)} \to S^{(2)}$, 
we set $\iota_{\rho}(s_{1}s_{2} \cdots) = \rho(s_{1}) \rho(s_{2}) \cdots  \in \Sigma(S^{(2)})$ 
for $s_{1}s_{2} \cdots \in \Sigma(S^{(1)})$. 
We assume that  
$T_{\rho} := \pi^{(2)} \circ \iota_{\rho} \circ (\pi^{(1)})^{-1}$ 
is a well-defined \textit{map}  from $K^{(1)}$ to $K^{(2)}$.

\begin{theorem}[Integration by substitution] 
\label{thm:substitution} 
Let $\phi^{(i)}$ be our integration on $K^{(i)}$. 
For any
functions $f,g$ on $K^{(2)}$ such that
the pair $(f,g)$ is integrable on $K^{(2)}$,   
we have the following. 
\begin{enumerate} 
\item 
Assume that $\rho$ preserves $0$ and $N-1$, respectively. 
Then a pair $(f \circ T_{\rho} , g \circ T_{\rho})$ is integrable on $K^{(1)}$ 
and we have $\phi^{(1)} (f \circ T_{\rho} , g \circ T_{\rho}) = \phi^{(2)} (f,g)$.
\item 
Assume that $\rho$ flips $0$ and $N-1$.  
Then a pair $(f \circ T_{\rho} , g \circ T_{\rho})$ is integrable on $K^{(1)}$ 
and we have $\phi^{(1)} (f \circ T_{\rho} , g \circ T_{\rho}) = - \phi^{(2)} (f,g)$. 
\end{enumerate} 

\begin{proof} 
Take $\bs{s} = (s_{1} , s_{2} , \cdots , s_{n}) \in (S^{(i)})^{\times n}$.  
By the definition of $\pi^{(i)}$, we have 
\begin{align*} 
\pi^{(i)}(s_{1} s_{2} \cdots s_{n} 00 \dots) 
&= f^{(i)}_{\bs{s}} (a^{(i)}) 
= a^{(i)}_{\bs{s}}  \\ 
\shortintertext{and}
\pi^{(i)}(s_{1} s_{2} \cdots s_{n} (N-1)(N-1) \dots) 
&= f^{(i)}_{\bs{s}} ( b^{(i)})   
= b^{(i)}_{\bs{s}}.  
\end{align*} 
The arguments are derived as follows: 
\begin{enumerate} 
\item 
If $\rho$ satisfies $\rho(0) = 0$ and $\rho(N-1) = N-1$, 
we have 
\begin{align*} 
T_{\rho}(a^{(1)}_{\bs{s}}) 
&= 
(\pi^{(2)} \circ \iota_{\rho} \circ (\pi^{(1)})^{-1}) (a^{(1)}_{\bs{s}}) \\ 
&= 
(\pi^{(2)} \circ \iota_{\rho} )(s_{1} s_{2} \cdots s_{n} 00 \dots) \\
&= 
\pi^{(2)} (\rho(s_{1}) \rho(s_{2}) \cdots \rho(s_{n}) 00 \cdots) \\ 
&= 
a^{(2)}_{\rho(\bs{s})} \\ 
\shortintertext{and}
T_{\rho}(b^{(1)}_{\bs{s}}) 
&= 
(\pi^{(2)} \circ \iota_{\rho} \circ (\pi^{(1)})^{-1}) (b^{(1)}_{\bs{s}}) \\ 
&= 
(\pi^{(2)} \circ \iota_{\rho} )(s_{1} s_{2} \cdots s_{n} (N-1)(N-1) \cdots) \\
&= 
\pi^{(2)} (\rho(s_{1}) \rho(s_{2}) \cdots \rho(s_{n}) (N-1)(N-1) \cdots) \\ 
&= 
b^{(2)}_{\rho(\bs{s})},   
\end{align*} 
where we denote $\rho(\bs{s}) = (\rho(s_{1}), \cdots , \rho(s_{n}))$. 
Thus we have 
\begin{align*} 
&\phantom{=}\ \phi^{(1)}_{n} (f \circ T_{\rho} , g \circ T_{\rho}) \\  
&= 
\sum_{\bs{s} \in S^{\times n}} \left( (f \circ T_{\rho})(a_{\bs{s}}^{(1)}) + f \circ T_{\rho}(b_{\bs{s}}^{(1)}) \right) \left( (g \circ T_{\rho})(b_{\bs{s}}^{(1)}) - (g \circ T_{\rho})(a_{\bs{s}}^{(1)}) \right) \\  
&= 
\sum_{\bs{s} \in S^{\times n}} \left( f(a_{\rho(\bs{s})}^{(2)}) + f(b_{\rho(\bs{s})}^{(2)}) \right) \left( g(b_{\rho(\bs{s})}^{(2)}) - g(a_{\rho(\bs{s})}^{(2)}) \right) \\  
&= 
\sum_{\bs{s} \in S^{\times n}} \left( f(a_{\bs{s}}^{(2)}) + f(b_{\bs{s}}^{(2)}) \right) \left( g(b_{\bs{s}}^{(2)}) - g(a_{\bs{s}}^{(2)}) \right) \\  
&= 
\phi^{(2)}_{n}(f,g) 
\end{align*} 
for any $n \in \mathbb{N}$. 
This proves (1). 
\item 
If $\rho$ satisfies $\rho(0) = N-1$ and $\rho(N-1) = 0$, 
we have 
\begin{align*} 
T_{\rho}(a^{(1)}_{\bs{s}}) 
&= 
(\pi^{(2)} \circ \iota_{\rho} \circ (\pi^{(1)})^{-1}) (a^{(1)}_{\bs{s}}) \\ 
&= 
(\pi^{(2)} \circ \iota_{\rho} )(s_{1} s_{2} \cdots s_{n} 00 \cdots) \\
&= 
\pi^{(2)} (\rho(s_{1}) \rho(s_{2}) \cdots \rho(s_{n}) (N-1)(N-1) \cdots) \\ 
&= 
b^{(2)}_{\rho(\bs{s})} \\ 
\shortintertext{and}
T_{\rho}(b^{(1)}_{\bs{s}}) 
&= 
(\pi^{(2)} \circ \iota_{\rho} \circ (\pi^{(1)})^{-1}) (b^{(1)}_{\bs{s}}) \\ 
&= 
(\pi^{(2)} \circ \iota_{\rho} )(s_{1} s_{2} \cdots s_{n} (N-1)(N-1) \cdots) \\
&= 
\pi^{(2)} (\rho(s_{1}) \rho(s_{2}) \cdots \rho(s_{n}) 00 \cdots) \\ 
&= 
a^{(2)}_{\rho(\bs{s})}.     
\end{align*}  
Thus we have 
\begin{equation*} 
\phi^{(1)}_{n} (f \circ T_{\rho} , g \circ T_{\rho})
= 
-\phi^{(2)}_{n}(f,g) 
\end{equation*} 
for any $n \in \mathbb{N}$. 
This proves (2). 
\end{enumerate} 
\end{proof} 

\end{theorem} 

\subsection{Moriyoshi and Natsume's example}
\label{subsec:MN-substitution}

We here recall a Young-type integration introduced by Moriyoshi and Natsume \cite{talk:Moriyoshi2013} with our notation,  
which is the first example of a Young-type integration on $CS$, 
and we clarify the integration is given as an example of Theorem \ref{thm:substitution}.  
Let $c \colon CS \to \mathbb{R}$ be the restriction on $CS$ 
of the Cantor function. 
In 
\cite{talk:Moriyoshi2013}, they essentially states that for any functions 
$f  , g   \in  C^{BV}$, where $C^{BV}$ is the continuous functions of bounded variation on $I = [0,1]$, 
the pair $(f \circ c, g \circ c)$ is integrable and the following holds: 
\begin{equation} 
\label{eq:MN's-formula}
\phi(f \circ c, g \circ c) = 2\int_{0}^{1} f(x)dg(x),  
\end{equation}  
where the right hand side is the Young integration on $I$. 
They also stated that $\phi$ satisfies the cyclic condition 
when $CS$ is replaced with its quotient space $CS/\{0, 1\}$. 
The induced cyclic $1$-cochain equals $0$ on $c^{\ast}C^{Lip}(I/\{ 0,1 \})$ 
and defines a non-trivial cyclic $1$-cocycle on $c^{\ast}C^{BV}(I/\{ 0,1 \})$. 

We next show that equation \eqref{eq:MN's-formula} can be considered as an example of Theorem \ref{thm:substitution}. 
Let $\mathcal{L}^{(1)} = \left( I , \{ 0,1 \} , \{ f^{(1)}_{0} , f^{(1)}_{1} \} \right)$ 
be the IFS of $CS$ described at the end of Section \ref{sec:prelim} and 
$\mathcal{L}^{(2)} = \left( I , \{ 0,1 \} , \{ f^{(2)}_{0} , f^{(2)}_{1} \} \right)$ 
the IFS of $I$ described in Example \ref{exm:interval}. 
Because of the equations 
\begin{align*} 
f^{(1)}_{0}\left( \sum_{i=1}^{n} \dfrac{2a_{i}}{3^{i}} \right) 
= 
\sum_{i=1}^{n} \dfrac{2a_{i}}{3^{i+1}}  \ 
\text{ and } \ 
f^{(1)}_{1}\left( \sum_{i=1}^{n} \dfrac{2a_{i}}{3^{i}} \right) 
= 
\dfrac{2}{3} + 
\sum_{i=1}^{n} \dfrac{2a_{i}}{3^{i+1}}   
\end{align*} 
for $a_{i} = 0$ or $1$, 
we have 
\begin{equation*} 
f^{(1)}_{s}\left( \sum_{i=1}^{n} \dfrac{2a_{i}}{3^{i}} \right) 
= \dfrac{2s}{3} +  \sum_{i=1}^{n} \dfrac{2a_{i}}{3^{i+1}}  
\quad (s=0,1). 
\end{equation*} 
Then, we get 
\begin{equation*} 
\pi^{(1)}(s_{1}s_{2}\cdots) = \sum_{i=1}^{\infty} \dfrac{2s_{i}}{3^{i}} 
\quad (s_{1}s_{2}\cdots \in \Sigma(\{ 0,1 \})). 
\end{equation*} 
Similarly, we have 
\begin{equation*} 
\pi^{(2)}(s_{1}s_{2}\cdots) = \sum_{i=1}^{\infty} \dfrac{s_{i}}{2^{i}} 
\quad (s_{1}s_{2}\cdots \in \Sigma(\{ 0,1 \})). 
\end{equation*} 
Recall that the set $CS$ is identified with the set 
\begin{equation*} 
\left\{ \sum_{i=1}^{\infty} \dfrac{2s_{i}}{3^{i}}  \,;\, s_{i} = 0,1 \right\}. 
\end{equation*} 
Thus, we have 
\begin{equation*} 
T_{\mathrm{id}} \left( \sum_{i=1}^{\infty} \dfrac{2s_{i}}{3^{i}} \right) 
= 
\sum_{i=1}^{\infty} \dfrac{s_{i}}{2^{i}}, 
\end{equation*} 
which means nothing but $T_{\mathrm{id}} = c$. 
Therefore, we have 
\begin{equation*} 
\phi^{(1)}(f \circ c , g \circ c) 
= 
\phi^{(1)}(f \circ T_{\mathrm{id}} , g \circ T_{\mathrm{id}}) 
= 
\phi^{(2)} (f,g) 
\end{equation*} 
when the pair $(f,g)$ is integrable on $I$ by Theorem \ref{thm:substitution} (1). 
For example, when $f$ and $g$ are continuous and in Wiener classes of order $p$ and $q$, respectively, 
with $1/p + 1/q > 1$, 
the Stieltjes integral $Y(f,g)$ exists in the Riemann sense; see Theorem \ref{thm:existence_stieltjes}. 
Then, the pair $(f,g)$ is integrable and we get 
\begin{equation*} 
\phi^{(1)}(f \circ c , g \circ c) 
= 
\phi^{(1)}(f \circ T_{\mathrm{id}} , g \circ T_{\mathrm{id}}) 
= 
\phi^{(2)} (f,g) 
= 
2\int_{0}^{1} f(x)dg(x) . 
\end{equation*} 
Note that the equation \eqref{eq:MN's-formula} is the case for $p=q=1$.

\begin{remark} 
By the arguments of this section,  
we can regard  
$T_{\rho}$ as a kind of generalizations of the Cantor function $c$. 
\end{remark}

\subsection{Examples on an Interval} 
\label{subsec:exm2-substitution}

In this section, we study other examples of Theorem \ref{thm:substitution}, in which IFSs have different similarity ratio. 
Let  
$\mathcal{L}^{(1)} = \left(I , \{ 0,1 \} , \{ f^{(1)}_{0} , f^{(1)}_{1} \} \right)$ 
be the IFS of $K^{(1)} = I$ described in Example \ref{exm:interval} 
and 
$\mathcal{L}^{(2)} = \left( I , \{ 0,1 \} , \{ f^{(2)}_{0} , f^{(2)}_{1} \} \right)$ 
the IFS of $K^{(2)} = I$ such that 
\begin{equation*} 
f^{(2)}_{0}(x) = \dfrac{1}{3}x \quad \text{and} \quad 
f^{(2)}_{1}(x) = \dfrac{2}{3}x + \dfrac{1}{3}. 
\end{equation*} 
When we denote $d_{i} = \sum_{k=1}^{i}s_{k} - 1$ for 
$s_{1}s_{2} \dots \in \Sigma (\{ 0,1 \})$, we have  
\begin{align*} 
f^{(2)}_{0}\left( \sum_{i=1}^{n} \dfrac{2^{d_{i}}}{3^{i}}s_{i} \right) 
= 
\sum_{i=1}^{n} \dfrac{2^{d_{i}}}{3^{i+1}}s_{i}   
\quad \text{and} \quad 
f^{(2)}_{1}\left( \sum_{i=1}^{n} \dfrac{2^{d_{i}}}{3^{i}}s_{i} \right) 
= 
\dfrac{1}{3} 
+ 
\sum_{i=1}^{n} \dfrac{2^{d_{i}+1}}{3^{i+1}}s_{i}.  
\end{align*} 
Thus we have 
\begin{equation*} 
f^{(2)}_{s}\left( \sum_{i=1}^{n} \dfrac{2^{d_{i}}}{3^{i}}s_{i} \right) 
= 
\dfrac{s}{3} 
+ 
\sum_{i=1}^{n} \dfrac{2^{s+d_{i}}}{3^{i+1}}s_{i}.  
\end{equation*} 
Therefore, we have 
\begin{equation*}
\pi^{(2)}(s_{1}s_{2}\cdots) 
= 
\sum_{i=1}^{\infty} \dfrac{2^{d_{i}}}{3^{i}}s_{i}. 
\end{equation*} 

Next, we prove that $T_{\rho}$ ($\rho = \mathrm{id}$ or transposition $(0 \, 1)$) induces a well-defined map from $K^{(1)}$ to $K^{(2)}$. 
For the purpose, it is only to show that we have 
\begin{equation*}
\pi^{(2)}(s_{1}s_{2} \cdots s_{n}1000 \cdots) 
=
\pi^{(2)}(s_{1}s_{2} \cdots s_{n}0111 \cdots) 
\quad (\bs{s} \in S^{\times n}).  
\end{equation*} 
It is clear by the following calculation: 
\begin{align*} 
\pi^{(2)}(s_{1}s_{2} \cdots s_{n}1000 \cdots) 
&= 
\sum_{i=1}^{n} \dfrac{2^{d_{i}}}{3^{i}}s_{i} + \dfrac{2^{d_{n}+1}}{3^{n+1}} \\ 
\pi^{(2)}(s_{1}s_{2} \cdots s_{n}0111 \cdots) 
&= 
\sum_{i=1}^{n} \dfrac{2^{d_{i}}}{3^{i}}s_{i} + \sum_{i=n+2}^{\infty} \dfrac{2^{d_{n} + i-n-1}}{3^{i}} \\ 
&= 
\sum_{i=1}^{n} \dfrac{2^{d_{i}}}{3^{i}}s_{i} + \dfrac{2^{d_{n}-1}}{3^{n}}\sum_{i=2}^{\infty} \dfrac{2^{i}}{3^{i}} \\ 
&= 
\sum_{i=1}^{n} \dfrac{2^{d_{i}}}{3^{i}}s_{i} + \dfrac{2^{d_{n}+1}}{3^{n+1}}. 
\end{align*} 
Then we have a well-defined map 
\begin{equation*} 
T_{\mathrm{id}}\left( \sum_{i=1}^{\infty} \dfrac{s_{i}}{2^{i}} \right) 
= 
\sum_{i=1}^{\infty} \dfrac{2^{d_{i}}}{3^{i}}s_{i} 
\quad \text{and} \quad 
T_{(0\,1)}\left( \sum_{i=1}^{\infty} \dfrac{s_{i}}{2^{i}} \right) 
= 
\sum_{i=1}^{\infty} \dfrac{2^{i - d_{i} - 2}}{3^{i}}(1-s_{i}).  
\end{equation*} 

Next, we show that 
two pairs $(1 \circ T_{\mathrm{id}} , x \circ T_{\mathrm{id}})$ 
and $(1 \circ  T_{(0 \, 1)} , x  \circ T_{(0\,1)})$ 
are integrable on $K^{(1)}$. 
By definition, we have 
\begin{align*} 
(x \circ T_{\mathrm{id}})(b_{\bs{s}}) - (x \circ T_{\mathrm{id}})(a_{\bs{s}}) 
&= 
\sum_{i=n+1}^{\infty} \dfrac{2^{d_{n} + i-n}}{3^{i}} 
= 
\dfrac{2^{d_{n}}}{3^{n}}\sum_{i=1}^{\infty}\dfrac{2^{i}}{3^{i}} 
= 
\dfrac{2^{d_{n} + 1}}{3^{n}} \\ 
\shortintertext{and} 
(x \circ T_{(0\,1)})(b_{\bs{s}}) - (x \circ T_{(0\,1)})(a_{\bs{s}}) 
&= 
-\sum_{i=n+1}^{\infty} \dfrac{2^{i - d_{n} - 2}}{3^{i}}  
-\dfrac{2^{n - (d_{n} + 1)}}{3^{n}}. 
\end{align*} 
Since the number of $s_{1}s_{2} \cdots s_{n}000 \cdots$ such that $d_{i}+1 = r$ 
equals $\binom{n}{r}$, 
we have 
\begin{align*} 
\sum_{\bs{s} \in S^{\times n}} \left( (x \circ T_{\mathrm{id}})(b_{\bs{s}}) - (x \circ T_{\mathrm{id}})(a_{\bs{s}})  \right) 
&= 
\sum_{r=0}^{n}\binom{n}{r}\dfrac{2^{r}}{3^{n}} 
= 1 \\ 
\sum_{\bs{s} \in S^{\times n}} \left( (x \circ T_{(0\,1)})(b_{\bs{s}}) - (x \circ T_{(0\,1)})(a_{\bs{s}}) \right) 
&= 
-\sum_{r=0}^{n}\binom{n}{r}\dfrac{2^{n-r}}{3^{n}} 
= -1. 
\end{align*} 
Therefore, we have 
$\phi^{(1)}_{n}(1 \circ T_{\mathrm{id}} , x \circ T_{\mathrm{id}}) = 2$ 
and 
$\phi^{(1)}_{n}(1 \circ  T_{(0\,1)} , x  \circ T_{(0\,1)}) = -2$ 
for any $n \in \mathbb{N}$. 
Taking $n \to \infty$, we know that two pairs $(1 \circ T_{\mathrm{id}} , x \circ T_{\mathrm{id}})$ 
and $(1 \circ T_{(0\,1)} , x  \circ T_{(0\,1)})$ 
are integrable on $K^{(1)}$ 
and its limit equals $2$ and $-2$, respectively. 
This implies the followings, which are examples of Theorem \ref{thm:substitution}: 
\begin{align*} 
\phi^{(1)}(1 \circ T_{\mathrm{id}} , x \circ T_{\mathrm{id}}) 
&= 2
= 2\int_{0}^{1}dx
= \phi^{(2)}(1,x), \\
\phi^{(1)}(1 \circ T_{(0\,1)} , x  \circ T_{(0\,1)})
&= -2
= -2\int_{0}^{1}dx
= - \phi^{(2)}(1,x).  
\end{align*} 

\subsection{A permutation for which integration by substitution does not hold}
\label{sec:not_hold_substitution} 

In order to study a permutation for which integration by substitution does not hold, 
we deal with a 5-adic Cantor set $CS_{2}$; see Appendix \ref{sec:generalized-Cantor-function}. 
Let $I = [0,1]$ and we set 
\[
f_{0}(x) = \dfrac{1}{5}x, \quad 
f_{1}(x) = \dfrac{1}{5}x + \dfrac{2}{5}, \quad 
f_{2}(x) = \dfrac{1}{5}x + \dfrac{4}{5}.  
\]
Then a tuple 
$
\mathcal{L} = \mathcal{L}^{(1)} = \mathcal{L}^{(2)} 
= \left( I  , S = \{ 0,1,2 \} , \{ f_{0} , f_{1} , f_{2}  \}  \right)
$  
defines an IFS, and we denote its self-similar set $CS_{2} = K^{(1)} = K^{(2)}$. 
Similar to the calculation in 
Section \ref{subsec:statement-substitution}, \ref{subsec:MN-substitution} and \ref{subsec:exm2-substitution}, 
we have 
\begin{align*} 
\pi (s_{1} s_{2} \cdots s_{n} 000 \cdots ) &= f_{\bs{s}}(0) =  a_{\bs{s}} \\ 
\pi (s_{1} s_{2} \cdots s_{n} 111 \cdots) &= f_{\bs{s}}(1/2) = \dfrac{a_{\bs{s}} + b_{\bs{s}}}{2} \\ 
\pi (s_{1} s_{2} \cdots s_{n} 222 \cdots ) &= f_{\bs{s}}(1) =  b_{\bs{s}} 
\end{align*} 
for $\bs{s} = (s_{1}, s_{2} , \cdots , s_{n}) \in S^{\times n}$. 
We also have 
\[
\pi (s_{1} s_{2} \cdots ) = \sum_{i=1}^{\infty} \dfrac{2s_{i}}{5^{i}} 
\quad (s_{1}s_{2}\cdots \in \Sigma(S)). 
\] 
Since the map $\pi : \Sigma(S) \to K$ is bijective, 
the composition $T_{\rho} = \pi \circ \iota_{\rho} \circ \pi^{-1}$ 
induces a well-defined map for any permutation $\rho \colon S \to S$. 

We study the case for the transposition $\rho = (1\,2)$. 
In this case, we have 
\begin{align*} 
T_{\rho}(a_{\bs{s}}) &= \pi (\rho(s_{1}) \rho(s_{2}) \cdots \rho(s_{n}) 000\cdots ) = a_{\rho(\bs{s})}  \\ 
\shortintertext{and}
T_{\rho}(b_{\bs{s}}) &= \pi (\rho(s_{1}) \rho(s_{2}) \cdots \rho(s_{n}) 111  \cdots ) 
= a_{\rho(\bs{s})} + \sum_{i=n+1}^{\infty} \dfrac{2}{5^{i}}. 
\end{align*}  
Let us define $g = c_{2,1/3}$; see Definition \ref{def:c_{k,p}-definition}. 
Then we have 
\begin{align*} 
(g \circ T_{\rho})(b_{\bs{s}}) - (g \circ T_{\rho})(a_{\bs{s}}) 
&= 
\sum_{i=1}^{n} \dfrac{\rho(s_{i})}{3^{i}} + \sum_{i=n+1}^{\infty} \dfrac{1}{3^{i}} - \sum_{i=1}^{n} \dfrac{\rho(s_{i})}{3^{i}} \\ 
&= 
\sum_{i=n+1}^{\infty} \dfrac{1}{3^{i}} 
= \dfrac{1}{2 \cdot 3^{n}}. 
\end{align*} 
Therefore, we have 
\begin{align*} 
\phi^{(1)}_{n}(1 \circ T_{\rho} , g \circ T_{\rho}) 
&= 
2\sum_{\bs{s} \in S^{\times n}} 
\left( (g \circ T_{\rho})(b_{\bs{s}}) - (g \circ T_{\rho})(a_{\bs{s}})  \right) \\ 
&= 
\sum_{\bs{s} \in S^{\times n}} \dfrac{1}{3^{n}} 
= 1. 
\end{align*}  
This implies that a pair $(1 \circ T_{\rho} , g \circ T_{\rho})$ 
is integrable on $K^{(1)}$ and its integral satisfies 
$\phi^{(1)}(1 \circ T_{\rho} , g \circ T_{\rho}) = 1$. 
By Theorem \ref{thm:integration-c_{k,p}} (1), we have 
\[
\phi^{(1)}_{n}(1 \circ T_{\rho} , g \circ T_{\rho}) 
= 
\dfrac{1}{2}
\phi^{(2)}_{n}(1,g) \quad (= 1 \neq 0).  
\]

On the other hand, 
we next define 
\[
h\left( \sum_{i=1}^{\infty} \dfrac{2s_{i}}{5^{i}} \right) 
= 
\sum_{i=1}^{\infty} \dfrac{(-1)^{s_{i}}s_{i}}{3^{i}} 
\]
on $CS_{2}$. By the following calculations 
\begin{align*} 
h(b_{\bs{s}}) - h(a_{\bs{s}}) 
&= 
\sum_{i=1}^{n} \dfrac{(-1)^{s_{i}}s_{i}}{3^{i}} + 2\sum_{i=n+1}^{\infty} \dfrac{1}{3^{i}} - \sum_{i=1}^{n} \dfrac{(-1)^{s_{i}}s_{i}}{3^{i}} \\ 
&= 
2 \sum_{i=n+1}^{\infty} \dfrac{1}{3^{i}} 
= 
\dfrac{1}{3^{n}}\\ 
\shortintertext{and} 
(h \circ T_{\rho})(b_{\bs{s}}) - (h \circ T_{\rho})(a_{\bs{s}}) 
&= 
\sum_{i=1}^{n} \dfrac{(-1)^{\rho(s_{i})}\rho(s_{i})}{3^{i}} - \sum_{i=n+1}^{\infty} \dfrac{1}{3^{i}} - \sum_{i=1}^{n} \dfrac{(-1)^{\rho(s_{i})}\rho(s_{i})}{3^{i}} \\ 
&= 
-\sum_{i=n+1}^{\infty} \dfrac{1}{3^{i}} 
= 
-\dfrac{1}{2 \cdot 3^{n}},  
\end{align*} 
we have 
\begin{align*} 
\phi^{(2)}_{n}(1   , h  ) 
&= 
2\sum_{\bs{s} \in S^{\times n}} \left( h(b_{\bs{s}}) - h(a_{\bs{s}}) \right) 
= 
2 \\ 
\shortintertext{and} 
\phi^{(1)}_{n}(1 \circ T_{\rho} , h \circ T_{\rho}) 
&= 
2\sum_{\bs{s} \in S^{\times n}}
\left( (h \circ T_{\rho})(b_{\bs{s}}) - (h \circ T_{\rho})(a_{\bs{s}})  \right) 
= 
-1.    
\end{align*} 
This indicates 
that the pair $(1,h)$ (resp. $(1 \circ T_{\rho} , h \circ T_{\rho})$)  
is integrable on $K^{(2)}$ (resp. $K^{(1)}$) 
and its integral satisfies 
$\phi^{(2)}(1   , h  ) = 2$
(resp. $\phi^{(1)}(1 \circ T_{\rho} , h \circ T_{\rho}) = -1$). 
Therefore, we have 
\begin{equation*}  
\phi^{(1)}(1 \circ T_{\rho} , h \circ T_{\rho}) 
= 
-\dfrac{1}{2} \phi^{(2)}(1   , h  ) \quad (= -1 \neq 0). 
\end{equation*} 

The above examples show that 
there is not a constant $c$ dependent only on $\rho = (1\,2)$ that 
makes  
$\phi^{(1)}(f \circ T_{\rho} , g \circ T_{\rho}) = c \phi^{(2)}(f,g)$ 
valid. 

\section{Integration by parts}
\label{sec:parts}

Our integration also inherits integration by parts. 
This property also has  
examples for integrable pairs $(f,g)$ (see Example \ref{exm:(x^l,x^mc)}). 

\begin{theorem}[Integration by Parts] 
\label{thm:properties-phi}   
For any functions $f$ and $g$ on $K$, we have 
$\phi_{n}(1 , fg) = \phi_{n}(f,g) + \phi_{n}(g,f)$. 
If two of $(1,fg)$, $(f,g)$ and $(g,f)$ are integrable, 
the rest is also integrable and 
we have $\phi(1,fg) = \phi(f,g) + \phi(g,f)$.  
\end{theorem} 

\begin{proof}  
$\phi_{n}(1, fg)$ can be written as follows: 
\begin{align} 
  \phi_{n}\left(1, fg\right) 
  &= \sum_{\bs{s} \in S^{\times n}} \operatorname{Tr}\left(\left[F, fg\right]M\right) 
  =  \sum_{\bs{s} \in S^{\times n}} \operatorname{Tr}\left(\left(f\left[F, g\right] + \left[F, f\right]g\right)M\right). \label{middle_fg}   
\end{align}  
Note that, for any interval $I = [a,b]$, 
we have the following relation:
\begin{align*}
gM_{I} 
&= \begin{bmatrix} g(a) & 0 \\ 0 & g(b) \end{bmatrix} \cdot \begin{bmatrix} 0 & -1 \\ 1 & 0 \end{bmatrix} \\
&= \begin{bmatrix} 0 & -1 \\ 1 & 0 \end{bmatrix} \cdot \begin{bmatrix} g(b) & 0 \\ 0 & g(a) \end{bmatrix} = M_{I} \begin{bmatrix} g(b) & 0 \\ 0 & g(a) \end{bmatrix}.
\end{align*}
Therefore, we obtain
\begin{align*} 
(\ref{middle_fg}) 
&= \sum_{\bs{s} \in S^{\times n}} \operatorname{Tr}\left(f\left[F, g\right]M\right) + \sum_{\bs{s} \in S^{\times n}} \operatorname{Tr}\left(\left[F, f\right]M \begin{bmatrix} g(b_{\bs{s}}) & 0 \\ 0 & g(a_{\bs{s}}) \end{bmatrix}\right) \\ 
&=  \phi_{n}\left(f, g\right) + \sum_{\bs{s} \in S^{\times n}} \operatorname{Tr}\left(\begin{bmatrix} g(b_{\bs{s}}) & 0 \\ 0 & g(a_{\bs{s}}) \end{bmatrix}\left[F, f\right]M\right) \\
&=  \phi_{n}\left(f, g\right) + \phi_{n}\left(g, f\right).
\end{align*}   
\end{proof}

By Theorem \ref{thm:properties-phi}, 
we obtain as follows.  

\begin{example} 
\label{exm:(x^l,x^mc)}
Let $l,m \geq 0$ be integers. 
Then we have 
\[
\phi(x^{l}, x^{m}c) = \phi(x^{l+m}, c). 
\]
\end{example}
\begin{proof}
By Theorem \ref{thm:properties-phi} (2), we have 
\[
\phi_{n}(1, x^{m}c) = \phi_{n}(x^{m}, c) + \phi_{n}(c, x^{m}). 
\]
By Example \ref{phi_cantor}, $\phi_{n}(x^{m}, c)$ converges to $\phi(x^{m}, c)$ as $n \rightarrow \infty$. On the other hand, $\phi_{n}(c, x^{m}) \to 0$ by Theorem \ref{thm:vanish} because $x^{m}$ is a Lipschitz function. Hence, $\phi_{n}(1, x^{m}c)$ converges to $\phi(x^{m}, c)$.  
Moreover, we obtain 
\[
\phi_{n}(1, x^{l+m}c) = \phi_{n}(x^{l}, x^{m}c) + \phi_{n}(x^{m}c, x^{l}).
\]
Since $\phi_{n}(x^{m}c, x^{l}) \to  0$ ($n \to \infty$), we obtain
\[
\phi (x^{l}, x^{m}c) 
= 
\lim_{n \rightarrow \infty}\phi_{n}(x^{l}, x^{m}c) 
= 
\lim_{n \rightarrow \infty}\phi_{n}(1, x^{l+m}c) 
= \phi(x^{l+m}, c).
\]

\end{proof}

\section{Term-by-term Integration} 
\label{sec:term-by-term}

\subsection{Term-by-integration for self-similar set}

We prove 
a theorem of term-by-term integration. 

\begin{theorem}[Term-by-term Integration] 
\label{thm:termbyterm} 
Let $0 < \alpha < 1$ and $0 < \beta < 1$ 
such that $\alpha + \beta > \dim_{S}(K)$. 
Suppose that uniformly bounded sequenses $\{ f_{m} \} \subset C^{\alpha}(K)$ and $\{ g_{m} \} \subset C^{\beta}(K)$ satisfy the following properties: 
\begin{enumerate}  
\item $f_{m} \to f$ and $g_{m} \to g$ ($m \to \infty$) uniformly on $K$. 
\item 
$(f,g)$ is integrable. 
\item 
A series 
$\sum_{n=1}^{\infty} | \psi_{n} (f_{m}, g_{m}) | $ 
converges uniformly in $m$. 
\label{enumerate:term-by-term-unif}
\end{enumerate} 
Then we have 
$\lim_{m \to \infty}\phi(f_{m}, g_{m}) 
= \phi(f,g)$. 
\end{theorem}

\begin{proof}   
By the multilinearity of $\phi$,  
it suffices to show that 
$\phi (f_{m} , g_{m})$ converges to $0$ 
when one of $f_{m}$, $g_{m}$ converges to $0$.  

Let $n,m \in \mathbb{N}$.  
As is the proof of Theorem \ref{thm:iff_condition}, 
we have 
\begin{align*} 
&\phantom{=} |\phi (f_{m}, g_{m}) - \phi_{n}(f_{m} , g_{m})| \\ 
&\leq 
\sum_{r=n}^{\infty} 
|\phi_{r+1} (f_{m}, g_{m}) - \phi_{r}(f_{m} , g_{m})| \\ 
&\leq  
2|f|_{\alpha}|g|_{\beta} (2N+1) (b-a) \sum_{r=n}^{\infty}
\left( \sum_{s=0}^{N-1} r_{s}^{\alpha + \beta} \right)^{r} 
+ \sum_{r=n}^{\infty} |\psi_{r+1}(f_{m}, g_{m})|  \\  
&= 2|f|_{\alpha}|g|_{\beta} (2N+1) \left( \sum_{s=0}^{N-1} r_{s}^{\alpha + \beta} \right)^{n} \sum_{r=0}^{\infty} \left( \sum_{s=0}^{N-1} r_{s}^{\alpha + \beta} \right)^{r}
+ \sum_{r=n}^{\infty} |\psi_{r+1}(f_{m}, g_{m})|   
\end{align*} 
Thus, for any $\epsilon > 0$, 
there exists $n_{0} > 0$ such that we have 
\[ 
|\phi (f_{m}, g_{m}) - \phi_{n}(f_{m} , g_{m})| 
< \epsilon 
\] 
for any $n \geq n_{0}$ and $m \in \mathbb{N}$. 
Then, the above calculations imply
\begin{align*} 
 |\phi(f_{m}, g_{m})| 
&\leq 
|\phi (f_{m}, g_{m}) - \phi_{n}(f_{m} , g_{m})|  +  |\phi_{n}(f_{m} , g_{m})| \\ 
&< 
\epsilon +  4N^{n} \| f_{m} \| \| g_{m} \|
\end{align*} 
for any $n \geq n_{0}$. 
Therefore, we have 
\[
\limsup_{m \to \infty} |\phi(f_{m},g_{m})| 
\leq \epsilon. 
\]
This completes the proof. 
\end{proof}

\subsection{Necessity of (\ref{enumerate:term-by-term-unif}) in Theorem \ref{thm:termbyterm}}
\label{sec:example_term_by_term}

We give examples in which term-by-term integration does not hold when a sequence of functions does not meet (\ref{enumerate:term-by-term-unif}) in Theorem \ref{thm:termbyterm}. 
We first set 
$g_{m}  = c_{1, m/(2m+1)}$ 
for any $m \in \mathbb{N}$; see Definition \ref {def:c_{k,p}-definition}.  
By Theorem \ref{thm:unif-c_k,a}, $g_{m} \to c_{1,1/2}$ as $m \to \infty$ 
uniformly on $CS_{1}$.  

Let $f_{m} = 1$ for any $m \in \mathbb{N}$. Then, the pair $(f, g)$ of the respective limit of $\{f_{m}\}$ and $\{g_{m}\}$ is integrable, see Section \ref{subsec:2k+1-adic}. However, the series $\sum_{n=1}^{\infty}|\psi_{n} (f_{m}, g_{m})|$ 
converges but not uniformly in $m \in \mathbb{N}$, which concludes that 
$\{ f_{m} \}$ and $\{ g_{m} \}$ do not meet Theorem \ref{thm:termbyterm} (3). 
In fact, since we have  
\begin{align*} 
\psi_{n+1}(f_{m}, g_{m}) 
&= 
\sum_{\bs{s} \in S^{\times n} } \mathrm{Tr}_{J_{\bs{s},0}}([F,g_{m}]M)  
= 
2\sum_{\bs{s} \in S^{\times n}} (g_{m}(a_{(\bs{s},1)}) - g_{m}(b_{(\bs{s},0)})) \\ 
&= 
2\left( 1 - \dfrac{m}{2m+1} \right)^{2} \sum_{\bs{s} \in S^{\times n}} \left( \dfrac{m}{2m+1} \right)^{n} \\ 
&= 
2\left( 1 - \dfrac{m}{2m+1} \right)^{2} \left( \dfrac{2m}{2m+1} \right)^{n},  
\end{align*} 
the series $\sum_{n=1}^{\infty}|\psi_{n}(f_{m}, g_{m})|$ 
converges for any $m \in \mathbb{N}$ 
but not uniformly in $m \in \mathbb{N}$.  
On the other hand, by Theorem \ref{thm:integration-c_{k,p}},  
the limit 
$\lim_{m \to \infty}\phi(f_{m}, g_{m}) = 0$ 
does not coincide with 
$\phi(f,g) = 2$.  
Therefore,  
the assumption (3) in Theorem \ref{thm:termbyterm} 
is essential in term-by-term integration.

\section{A comment on non-continuous case}  
\label{sec:comment-non-conti}

The Young integration introduced in \cite{MR1555421} is defined on the space of Wiener class $W_{p}$. 
The class $W_{p}$ is a set of functions whose $p$-variation is bounded, and in general can include non-continuous functions. 
Our integral can be calculated
for some non-continuous functions as follows:  
Set 
\begin{align*}
h(x) = \begin{dcases*}
    $0$  & for $0 \leq x < \frac{1}{3}$, \\
    $1$  & for $\frac{1}{3} \leq x \leq 1$ 
\end{dcases*}  
\end{align*} 
on $CS$. 
The function $h$ is obviously non-continuous at $x = \frac{1}{3}$, 
and for each $\bs{s} = (s_{1}, \dots , s_{n}) \in S^{\times n}$, we have
\begin{align*}
\mathrm{Tr}_{I_{\bs{s}}}(h[F, h]M) 
= \begin{dcases*}
    $1$ & for $s_{1} = 0$ \text{and} $s_{2} = \dots = s_{n} = 1$, \\
    $0$ & \text{otherwise}.
\end{dcases*}
\end{align*} 
Therefore, because of $\phi_{n}(h, h) = 1$ for any  $n \in \mathbb{N}$,  
we can conclude that the pair $(h, h)$ is integrable and 
we have $\phi(h,h) = 1$. 

To our knowledge there is no proper definition of 
``$p$-variation'' for the Cantor set yet. 
We leave the extension of our integration to the setting of `Wiener class' for the future study.

\appendix 

\section{The Young integration on the interval} 
\label{app:Young}

We recall the definition of the Young integration and its existence by referring to \cite{MR1555421}. Let $I$ be an interval $[ a, b ]$.  

\begin{definition}    
\label{def:stieltjes_integral}
Let $f, g: I \rightarrow \mathbb{R}$. We say that the Stieltjes integral
\[ 
Y(f,g) 
= 
\int_{a}^{b} f(x) d g(x) 
\] 
exists in the Riemann sense with the value $J$, if the sum $\sum_{r=1}^{N} f(\xi_{r})\left( g(x_{r}) - g(x_{r-1})\right)$, in which $a=x_{0} \leq \xi_{1} \leq x_{1} \leq \cdots \leq x_{N-1} \leq \xi_{N} \leq x_{N} = b$, differs from $J$ by at most $\epsilon_{\delta}$ in modulus, when all $[ x_{r-1}, x_{r}]$ have lengths less than $\delta$, where $\epsilon_{\delta} \rightarrow 0$ as $\delta \rightarrow 0$.
\end{definition}

We associate with $f(x): I \rightarrow \mathbb{R}$ and with $p > 0$, $\delta > 0$, the quantity
\[V_{p}^{(\delta)}(f) = V_{p}^{(\delta)}(f;a, b) = \sup_{|\chi| \leq \delta}\left( \sum_{r} \left|f(x_{r}) - f(x_{r-1})\right|^{p}\right)^{\frac{1}{p}}\]
for all subdivisions $\chi$ of $[ a , b ]$ into parts $[ x_{r-1}, x_{r} ]$ of lengths each less than $\delta$. We write $V_{p}(f)$ as the value of $V_{p}^{(\delta)}(f)$ for $\delta$ exceeding $b - a$ and call $V_{p}(f)$ \textit{the mean variation of order} $p$. We shall say that $f$ belongs to the Wiener class $W_{p}$  
if $V_{p}(f)$ is finite. 
Note that if $f$ is $1/p$-H\"older continuous function on $I$ then 
we have $f \in W_{p}$.  

\begin{theorem}[Existence on Stieltjes integral]
\label{thm:existence_stieltjes}
Let $f \in W_{p}$ and $g \in W_{q}$. If $\frac{1}{p} + \frac{1}{q} > 1$ and $f, g$ have no common discontinuities, their Stieltjes integral exists in the Riemann sense.
\end{theorem}

\section{Generalization of the Cantor function on the $(2k+1)$-adic Cantor set} 
\label{sec:generalized-Cantor-function}

For a positive integer $k \in \mathbb{N}$ and $s \in S = \{ 0,1,\dots k \}$, 
we set  
\[
f_{s}(\bs{x}) = \dfrac{x}{2k+1} + \dfrac{2s}{2k+1} 
\quad (x \in \mathbb{R}). 
\]
Then we call the self-similar set $CS_{k}$ of the IFS $\left(I = [0,1] , S , \{ f_{s} \}_{s \in S} \right)$ the $(2k+1)$-adic Cantor set. 
By definition, we have  $\dim_{S}(CS_{k}) = \log_{(2k+1)}(k+1)$. 
It is easy to show that any element $x \in CS_{k}$ is given by the formula 
\[
x = \sum_{i=1}^{\infty} \dfrac{2s_{i}}{(2k+1)^{i}} 
\quad (s_{i} \in S). 
\]
In particular, we have 
\[
a_{\bs{s}}
= 
\sum_{i=1}^{n} \dfrac{2s_{i}}{(2k+1)^{i}}
\quad \text{and} \quad 
b_{\bs{s}} 
= 
a_{\bs{s}} + \sum_{i=n+1}^{\infty} \dfrac{2k}{(2k+1)^{i}} \quad (\bs{s} \in S^{\times n}). 
\] 
See Section \ref{sec:prelim} for the definition of $a_{\bs{s}}$ and $b_{\bs{s}}$. 
We want to generalize the Cantor function on the middle-third Cantor set $CS = CS_{1}$ to $CS_{k}$. 
The following definition is almost the same as the one in \cite{MR3713565}, but slightly different. 

\begin{definition} 
\label{def:c_{k,p}-definition} 
For $0 < p < 1$, we set 
\[
c_{k,p}\left( \sum_{i=1}^{\infty} \dfrac{2s_{i}}{(2k+1)^{i}} \right) 
= 
\dfrac{1-p}{k} \sum_{i=1}^{\infty}s_{i}p^{i-1} 
\quad (s_{i} \in S). 
\] 
\end{definition} 

A straightforward calculation implies the following basic properties: 

\begin{proposition} 
\label{prop:c_{k,a}-property} 
For $0 < p < 1$, we have the followings: 
\begin{enumerate} 
\item 
$c_{1,1/2} = c$, the Cantor function on the middle-third Cantor set. 
\label{item:c_{1,1/2}}
\item 
$c_{k,1/(2k+1)} = x|_{CS_{k}}$.  
\label{item:c_{k,1/(2k+1)}}
\item 
$c_{k,p}(0) = 0$. 
\item 
$c_{k,p}\left( \frac{1}{(2k+1)^{n}} \right) 
= p^{n} $ 
for any non-negative integer $n=0,1,2,\dots$. 
In particular, we have $c_{k,p}(1) = 1$. 
\end{enumerate} 
\end{proposition}

Moreover, $c_{k,p}$ satisfies the following properties:  

\begin{theorem} 
\label{thm:c_{k,a}-Holder} 
For $0 < p < 1$,  
$c_{k,p}$ is a $\log_{(2k+1)}(p^{-1})$-H\"{o}lder continuous function 
on $CS_{k}$.  
\end{theorem} 

\begin{proof} 
Take any $x,y \in CS_{k}$ 
with $x = \sum_{i=1}^{\infty} \frac{2s_{i}}{(2k+1)^{i}}$ 
and $y = \sum_{i=1}^{\infty} \frac{2t_{i}}{(2k+1)^{i}}$ 
such that 
\[ 
\dfrac{1}{(2k+1)^{n+1}} < |x-y| \leq \dfrac{1}{(2k+1)^{n}}
\] 
for a non-negative integer $n$. 
Then we have $s_{i} = t_{i}$ ($i = 1,2,\dots n$) and $s_{n+1} = t_{n+1}$, and 
\begin{align*} 
|c_{k,p}(x) - c_{k,p}(y)| 
&= 
\dfrac{1-p}{k}
\left|  \sum_{i=n+1}^{\infty}  (s_{i} - t_{i})p^{i-1} \right|   
\leq 
(1-p)
\sum_{i=n+1}^{\infty} p^{i-1}  \\ 
&= 
p^{n} 
= 
p^{-1} \left( \dfrac{1}{(2k+1)^{n+1}} \right)^{\log_{(2k+1)}(p^{-1})} \\ 
&\leq 
p^{-1} |x-y|^{\log_{(2k+1)}(p^{-1})}. 
\end{align*} 
This implies that the function $c_{k,p}$ is a H\"{o}lder continuous function 
with the H\"{o}lder exponent $\log_{(2k+1)}(p^{-1})$. 
\end{proof}

\begin{theorem} 
\label{thm:unif-c_k,a}
For any fixed $0 < p_{0} < 1$, 
$c_{k,p}$ converges to $c_{k,p_{0}}$ as $p \to p_{0}$ 
uniformly on $CS_{k}$. 
\end{theorem} 

\begin{proof} 
Fix any $\epsilon_{0} > 0$ such that $\epsilon_{0} < p_{0} < 1 - \epsilon_{0}$. 
Let $\epsilon > 0$ such that $\epsilon < \epsilon_{0}$. 
For any $x = \sum_{i=1}^{\infty} \frac{2s_{i}}{(2k+1)^{i}} \in CS_{k}$ 
and $0 < p < 1$ such that $|p - p_{0}| < \epsilon$, 
we have 
\begin{align*}
    |c_{k,p}(x) - c_{k,p_{0}}(x)| 
    &\leq 
    \dfrac{1}{k}\sum_{i=1}^{\infty} s_{i}|(1-p)p^{i} - (1-p_{0})p_{0}^{i}| \\ 
    &\leq 
    \sum_{i=1}^{\infty} \left( |p^{i} - p_{0}^{i}| + |p^{i+1} - p_{0}^{i+1}| \right) \\ 
    &= 
    |p-p_{0}|\sum_{i=1}^{\infty} \left( \sum_{k=0}^{i-1}p^{k}p_{0}^{i-1-k} + \sum_{k=0}^{i}p^{k}p_{0}^{i-k} \right) \\ 
    &\leq  
    2|p - p_{0}| \sum_{i=1}^{\infty}   i (p_{0} + \epsilon_{0})^{i-1}  \\ 
    &\leq  
    2|p - p_{0}| \dfrac{1}{(1 - (p_{0} + \epsilon_{0}))^{2}} \\ 
    &< 
    \dfrac{2\epsilon}{(1 - (p_{0} + \epsilon_{0}))^{2}} . 
\end{align*} 
This proves the theorem. 
\end{proof} 

\subsection*{Acknowledgments} Tatsuki Seto was supported by JSPS KAKENHI Grant Number  21K13795.

\bibliographystyle{plain}
\bibliography{ref_Young_CS}

\end{document}